\documentclass{amsart}
\usepackage{xcolor}
\usepackage{mathtools}
\usepackage{amsmath}
\usepackage{amsthm}
\usepackage{graphicx}
\usepackage{amssymb}
\usepackage{epstopdf}
\usepackage{nicefrac}
\usepackage{mathrsfs}
\usepackage{mathrsfs}
\usepackage{hyperref}
\usepackage{enumitem} 
\usepackage{xcolor}

\usepackage{tikz}  
\usetikzlibrary{calc,arrows.meta,decorations.pathreplacing}

\theoremstyle{plain}
\newtheorem{theo}{Theorem}[section]
\newtheorem{lm}[theo]{Lemma}
\newtheorem*{defin}{Definition}
\newtheorem{cor}[theo]{Corollary}
\newtheorem{prop}[theo]{Proposition}

\theoremstyle{plain}

\DeclareMathOperator{\diag}{diag}

\DeclareMathOperator{\sgn}{sgn}
\DeclareMathOperator{\Ha}{H}

\DeclareMathOperator{\PW}{PW} 
\DeclareMathOperator{\supp}{supp} 
\DeclareMathOperator{\spans}{span} 
\DeclareMathOperator{\Real}{Re} 
\DeclareMathOperator{\ext}{ext} 
\DeclareMathOperator{\dist}{dist} 
\DeclareMathOperator{\Imaginary}{Im} 
\title[Hankel operators on Paley--Wiener spaces]{Besov spaces and Schatten class Hankel operators for Hardy and Paley--Wiener spaces in higher dimensions}

\author[K. Bampouras]{Konstantinos Bampouras}
\address{Department of Mathematical Sciences, Norwegian University of Science and Technology (NTNU), 7491 Trondheim, Norway}
\email{konstantinos.bampouras@ntnu.no}

\author[K.-M. Perfekt]{Karl-Mikael Perfekt}
\address{Department of Mathematical Sciences, Norwegian University of Science and Technology (NTNU), 7491 Trondheim, Norway}
\email{karl-mikael.perfekt@ntnu.no}

\date{\today}
\usepackage{cite}
\begin{document}
	\begin{abstract}
		We consider Schatten class membership of Hankel operators on Paley--Wiener spaces of convex $\Omega \subset \mathbb{R}^n$, both for bounded and unbounded domains. In particular, the classical product Hardy spaces fit within our theory. For admissible domains, we develop a framework and theory of Besov spaces of Paley--Wiener type, and prove that a Hankel operator belongs to the Schatten class $S^p$ if and only if its symbol belongs to a corresponding Besov space, for $1 \leq p \leq 2$. We extend this result to all $1 \leq p < \infty$ for the classical product Hardy spaces and to $1 \leq p < 2n/(n-1)$ for the Paley--Wiener space of a bounded smooth domain $\Omega \subset \mathbb{R}^n$ of strictly positive curvature.
	\end{abstract}
	
	\maketitle
	
	\section{Introduction}
	A classical Hankel operator $H_\varphi \colon H^2(\mathbb{R}) \to H^2(\mathbb{R})$ is defined by the equation
	$$\widehat{H_\varphi f}(x) = \int_0^\infty \widehat{\varphi}(x+y) \hat{f}(y) \, dy, \qquad x > 0.$$
	Here $\hat{f} = \mathcal{F} f$ denotes the Fourier transform, and $H^2(\mathbb{R}) = \mathcal{F}^{-1} L^2(\mathbb{R}_+)$ is the Hardy space of the upper half-plane. By transfer to the Hardy space $H^2(\mathbb{T})$ of the disc, it is well known that $H_\varphi$ is unitarily equivalent to a Hankel matrix operator $\{a_{n+m}\}_{n,m \geq 0}$. Nehari's theorem states that $H_\varphi \colon H^2(\mathbb{R}) \to H^2(\mathbb{R})$ is bounded if and only if there is a bounded function $\psi \in L^\infty(\mathbb{R})$ such that $H_\varphi = H_\psi$.  Several proofs of this key theorem in the theory of Hardy spaces can be found in \cite{Nik02}.  Schatten class membership, $H_\varphi \in S^p$, $0 < p < \infty$, was later described by Peller \cite{MR1949210}, in terms of Besov spaces. See also \cite{MR674875} for intrinsic results in the language of $H^2(\mathbb{R})$.
	
	In the multi-parameter setting, understanding the boundedness of (\textit{small}) Hankel operators $H_\varphi \colon H^2(\mathbb{R} \times \cdots \times \mathbb{R}) \to H^2(\mathbb{R} \times \cdots \times \mathbb{R})$, 
	$$\widehat{H_\varphi f}(x) = \int_{\mathbb{R}_+^n} \widehat{\varphi}(x+y) \hat{f}(y) \, dy, \qquad x \in \mathbb{R}_+^n,$$
	is well-recognized as one of the most important endeavors in higher-dimensional Hardy space theory and multi-parameter harmonic analysis. Note that the product Hardy space 
	$$H^2(\mathbb{R} \times \cdots \times \mathbb{R}) \simeq H^2(\mathbb{T}^n)$$
	can be understood as the inverse Fourier transform of $L^2(\mathbb{R}^n_+)$, but that proving Nehari's theorem in this context remains an unsolved problem despite persistent efforts over decades \cite{ChaFef85, FergLac02, FergSad00, HTV21}. Characterizing the Schatten class membership of multi-parameter operators therefore presents an intriguing problem, see \cite{MR924766, MR1001657}. In \cite{MR2097606}, it was observed that the vector-valued version of Peller's Schatten class characterization can be iterated to describe the operators $H_\varphi \in S^p(H^2(\mathbb{R} \times \cdots \times \mathbb{R}))$, when $1 < p < \infty$.
	Much more recently, Schatten membership was described for a class of multi-parameter dyadic paraproducts related to \textit{big} Hankel operators \cite{LLW23} -- however, non-zero big Hankel operators cannot be compact when $n > 1$ \cite{MR1415032}.
	
	A relatively simple special case of our theory recovers and extends the characterization of the Schatten classes for \textit{small} Hankel operators to the range $1\leq p<\infty$. To obtain a complete description, we must introduce special Besov spaces $B_{p,q}^{s}(\mathbb{R}^n_+)$ adapted to functions with Fourier support in $\mathbb{R}^n_+$; the details can be found in Sections \ref{sec:besovspace} and \ref{sec:poly}. We will work with a class of weighted Hankel operators $H^{\sigma, \tau}_\varphi$, $\sigma, \tau \in \mathbb{R}$, defined by the equation
	$$\widehat{\Ha_{\varphi}^{\sigma,\tau}f}(x)= \int_{\mathbb{R}^n_+} \widehat{\varphi}(x+y)\widehat{f}(y) x^\sigma y^\tau \, dy, \qquad x \in \mathbb{R}^n_+, \; x^\sigma = x_1^\sigma \ldots x_n^{\sigma}.$$
	\begin{theo}\label{Hardyintro}
		Let $1\leq p<\infty$ and $\sigma,\tau>\max(-\frac{1}{2},-\frac{1}{p})$. The extended Hankel operator $\Ha_\varphi^{\sigma, \tau} \colon H^2(\mathbb{R} \times \cdots \times \mathbb{R}) \to H^2(\mathbb{R} \times \cdots \times \mathbb{R})$ belongs to the Schatten class $S^{p}$ if and only if $\varphi\in B_{p,p}^{\sigma+\tau+\frac{1}{p}}(\mathbb{R}^n_+)$.
	\end{theo}
	
	The main goal of this paper is to investigate the Schatten class membership of Hankel operators
	$\Ha_{\varphi} \colon \PW(\Omega) \to \PW(\Omega)$ on the Paley--Wiener space 
	$$\PW(\Omega) = \{f \in L^2(\mathbb{R}^n) \, : \, \supp \hat{f} \subset \overline{\Omega}\}$$
	associated with a convex domain $\Omega \subset \mathbb{R}^n$ which does not contain lines,
	$$\widehat{\Ha_{\varphi}f}(x)=\int_{\Omega}\widehat{\varphi}(x+y)\widehat{f}(y) \, dy, \qquad x \in \Omega.$$
	In particular, small Hankel operators on the Hardy space fit into the framework with $\Omega = \mathbb{R}^n_+$.
	
	The case when $\Omega = (-1,1) \subset \mathbb{R}$ is an interval was given in-depth treatment by Peller \cite{Peller88} and Rochberg \cite{MR878246}, in particular, describing when $\Ha_\varphi \in S^p(\PW((-1,1)))$, $0 < p < \infty$, in terms of a Besov space condition. For $1 \leq p < \infty$ and $n > 1$, Peng \cite{MR1001657} gave a characterization of the Schatten class operators $\Ha_\varphi \in S^p(\PW(C_n))$ for the Paley--Wiener space of the cube $C_n = (-1,1)^n$. Many classical ideas concerning Schatten class membership of integral operators are also applicable, ideas found for example in the papers of Janson and Peetre \cite{MR924766} and Timotin \cite{MR837521}, and we wish to use these in our study of general domains $\Omega$.
	
	However, it is a major challenge to find the correct definition of the Paley--Wiener type Besov spaces $B^{\frac{1}{p}}_{p,p}(2\Omega)$, and to develop a satisfactory theory of these which, in particular, must include basic multiplier results. Unlike when $\Omega = \mathbb{R}^n$, or when $\Omega = (0,\infty)$, the instrumental case for $\Omega = (-1,1)$, we are not able to exploit a homogeneous structure of $\Omega$. Instead, our scale of Besov spaces must be adapted to the geometry of $\Omega$, as well as to the weight
	\[
	\omega_{\Omega}(x)= m(\Omega\cap(2x-\Omega)) \qquad x \in \Omega.
	\]
	The sets $M_\Omega(x) = \Omega\cap (2x-\Omega)$ are known as Macbeath regions. $M_\Omega(x)$ is the largest subset of $\Omega$ that is symmetric with respect to $x$, and has previously been used to study boundary behavior and approximation of convex bodies.  In particular, for bounded sets, the asymptotic behavior of $\omega_\Omega(x) = m(M_\Omega(x))$ is well understood \cite{MR2327289,MR1194970,SW90}.
	
	Accordingly, we consider a class of weighted extended Hankel operators
	$\Ha_{\varphi}^{\sigma,\tau}:\PW(\Omega)\to \PW(\Omega)$, given by
	$$\widehat{\Ha_{\varphi}^{\sigma,\tau}f}(x)=\int_\Omega \widehat{\varphi}(x+y)\widehat{f}(y)\omega_{\Omega}^{\sigma}(x)\omega_{\Omega}^{\tau}(y) \, dy, \qquad x \in \Omega.$$
	Our second main result reads as follows. Note that smooth domains with strictly positive curvature are known as strongly convex domains.
	\begin{theo}\label{thm:mainsmooth}
		Let $\Omega \subset \mathbb{R}^n$ be a bounded convex $C^2$-domain with non-vanishing boundary curvatures. Then, 
		\begin{enumerate}
			\item For every $1\leq p<\infty$ and $\sigma,\tau>-\frac{1}{n+1}+\max(0,\frac{1}{2}-\frac{1}{p})$ with $\sigma+\tau>-\frac{2}{p(n+1)}+\frac{2(n-1)}{n+1}\max(0,\frac{1}{2}-\frac{1}{p})$, $\Ha_\varphi^{\sigma,\tau}$ is in the Schatten class $S^{p}(\PW(\Omega))$ if and only if $\varphi\in B_{p,p}^{\sigma+\tau+\frac{1}{p}}(2\Omega)$.
			\item In particular, for every $1\leq p<2\frac{n}{n-1}$, $\Ha_\varphi$ is in the Schatten class $S^{p}(\PW(\Omega))$ if and only if $\varphi\in B_{p,p}^{\frac{1}{p}}(2\Omega)$.
		\end{enumerate}
	\end{theo}
	We also remark that in the disc setting $\Omega = \mathbb{D}$, Schatten classes for Hankel operators were previously considered in the smaller range $1 \leq p \leq 2$ \cite{MR953994}.
	
	To construct the Besov spaces $B_{p,q}^{s}(\Omega)$ of Paley--Wiener type, we introduce a notion of \textit{admissible} domains $\Omega$. Roughly speaking, an admissible domain $\Omega$ admits an open cover $\{A^i_j\}$ of parallelepipeds, each piece $A^i_j$ contained in the level set $\omega_{\Omega}^{-1}((a^{j}, a^{j+1}))$, $j\in\mathbb{Z}$, and of size $m(A^i_j) \approx a^{j}$. Furthermore, the cover must satisfy some natural geometrical properties. See Section~\ref{sec:besovspace} for details.
	
	Actually, it turns out that every convex domain not containing lines is admissible. This has been demonstrated in \cite{BampourasPerfekt2026}, around two years after the initial preprint publication of this article. However, as we do not prove this general admissibility result in this article, we continue to state our results in terms of admissible domains.
	
	The properties of the cover $\{A^i_j\}$ ensure that there exists a subordinate partition of unity $(\widehat{\varphi}^i_j)$ of $\Omega$ such that $\sup_{j,i} \|\varphi^i_j\|_{L^1(\mathbb{R}^n)} < \infty$. The Besov space $B_{p,q}^{s}(\Omega)$ is then defined with respect to any such partition of unity,
	$$\|f\|_{B_{p,q}^{s}(\Omega)}=\left(\sum_{j,i}a^{qjs}\|f\ast\varphi^{i}_{j}\|^{q}_{L^{p}}\right)^{\frac{1}{q}}.$$
	Since $\omega_{\Omega}(x) \approx a^j$ for $x \in A^i_j$, this definition is indeed adapted to the weight $\omega_{\Omega}$.
	
	In the range $1 \leq p \leq 2$, we will demonstrate sufficiency of the Besov space condition for a general admissible domain. 
	\begin{theo}\label{thm:mainsuf}
		Let $\Omega$ be an admissible subset of $\mathbb{R}^{n}$, and let $1\leq p\leq 2$. Suppose that $\varphi\in B_{p,p}^{\sigma+\tau+\frac{1}{p}}(2\Omega)$ for some $\sigma,\tau>-\frac{1}{n+1}$ with $\sigma+\tau>-\frac{2}{p(n+1)}$. Then $\Ha_{\varphi}^{\sigma,\tau}\in S^{p}(\PW(\Omega))$. Furthermore, if $\Omega$ is a polyhedron the same holds true for $\sigma,\tau>-\frac12$ with $\sigma+\tau>-\frac1p$.
	\end{theo}
	
	On the other hand, we are able to prove the necessity of the Besov space condition in full generality.
	\begin{theo}\label{thm:mainnec}
		Let $\Omega$ be an admissible subset of $\mathbb{R}^{n}$, and let $1 \leq p \leq \infty$. Suppose that $\Ha^{\sigma,\tau}_{\varphi}\in S^{p}(\PW(\Omega))$ for some $\sigma,\tau\in\mathbb{R}$. Then $\varphi\in B_{p,p}^{\sigma+\tau+\frac{1}{p}}(2\Omega)$.
	\end{theo}
	
	Combining Theorem~\ref{thm:mainsuf} and Theorem~\ref{thm:mainnec} we obtain a characterization for $1 \leq p \leq 2$, valid for general admissible domains $\Omega$.
	\begin{cor}\label{cor:mainchar}
		Let $\Omega$ be an admissible subset of $\mathbb{R}^{n}$. Then, for $1 \leq p \leq 2$ and $\sigma,\tau>-\frac{1}{n+1}$ with $\sigma+\tau>-\frac{2}{p(n+1)}$, $$\Ha_{\varphi}^{\sigma,\tau}\in S^{p}(\PW(\Omega))\text{   if and only if  }\varphi\in B_{p,p}^{\sigma+\tau+\frac{1}{p}}(2\Omega),$$
		and the corresponding norms are equivalent. Furthermore, if $\Omega$ is a polyhedron the same holds true for $\sigma,\tau>-\frac12$ with $\sigma+\tau>-\frac1p$.
	\end{cor}
	For simple polytopes $P$, our framework also lets us characterize the extended Hankel operators $\Ha_{\varphi}^{\sigma, \tau} \in S^{p}(\PW(P))$, $1 \leq p < \infty$, where $\sigma,\tau>\max(-\frac{1}{2},-\frac{1}{p})$ and $\sigma+\tau>-\frac1p$, see Theorem~\ref{thm:mainpoly}. Recall that a polytope in $\mathbb{R}^n$ is simple if each of its vertices is adjacent to exactly $n$ edges. 
	
	This paper is organized as follows. In Section~\ref{sec:besovspace} we lay out the fundamental properties of the weight function $\omega_\Omega$, and we develop our framework for Besov spaces $B^s_{p,q}(\Omega)$. In Section~\ref{sec:suf} we establish sufficient conditions for Schatten class membership of $\Ha_{\varphi}^{\sigma,\tau}$ when $p \in \{1, 2, \infty\}$, which in conjunction with interpolation yields results for $1 \leq p < \infty$. In Section~\ref{sec:nec} we prove the necessity of the Besov space condition. In Section~\ref{sec:smooth} we treat smooth domains with positive curvature, and in Section~\ref{sec:poly} we consider the product Hardy spaces and Paley--Wiener spaces of polytopes. 
	
	\subsection*{Notation} As usual, we write that $a \lesssim b$ if there is a constant $K > 0$ (possibly depending on parameters understood from context) such that $a \leq Kb$, and $a \approx b$ if additionally $b \lesssim a$. In the same vein we will reserve the letters $C$ and $K$ for constants, the values of which may vary from line to line.
	
	\subsection*{Acknowledgments} The second author was supported by grant no. 334466 of the Research Council of Norway, “Fourier Methods and Multiplicative Analysis”. AI was used to identify some mistakes in an earlier version of this paper. It also suggested the improved proof of Theorem~\ref{S2}.
	
	\section{Besov spaces of Paley--Wiener type} \label{sec:besovspace}
	\subsection*{The weight function associated with $\Omega$}
	The appropriate weight to consider in the construction of our Besov spaces is readily identified by considering the usual characterization of Hilbert--Schmidt integral operators. If $\Ha_\varphi \in S^2$, $\hat{\varphi}$ must be a measurable function in $2\Omega$, and 
	\[
	\|\Ha_\varphi\|_{S^2}^2 = \int_\Omega \int_\Omega |\hat{\varphi}(x+y)|^2 \, dx \, dy = \int_{2\Omega} |\hat{\varphi}(x)|^2 m(\Omega\cap(x-\Omega)) \, dx,
	\]
	where $m$ denotes the $n$-dimensional Lebesgue measure. Accordingly, we introduce the weight 
	\[
	\omega_{\Omega}(x)= m(\Omega\cap(2x-\Omega)), \qquad x \in \mathbb{R}^n.
	\]
	Since $\Omega$ does not contain lines, it is not difficult to see that $\omega_{\Omega}(x)$ is finite everywhere. We will consider the following class of weighted Hankel operators.
	\begin{defin}
		Let $\widehat{\varphi}$ be a distribution in $2\Omega$. For $\sigma,\tau\in\mathbb{R}$, we define the extended Hankel operator $\Ha_{\varphi}^{\sigma,\tau}:\PW(\Omega)\to \PW(\Omega)$ by 
		$$\widehat{\Ha_{\varphi}^{\sigma,\tau}f}(x)=\int\widehat{\varphi}(x+y)\widehat{f}(y)\omega_{\Omega}^{\sigma}(x)\omega_{\Omega}^{\tau}(y)\chi_{\Omega}(x)\chi_{\Omega}(y)dy.$$
		This is always a densely defined operator, well-defined at least for $f \in S(\mathbb{R}^n)$ with $\supp \hat{f} \subset \Omega$ compact.
	\end{defin}
	There is a deliberate abuse of notation in this definition which will not cause any issues. Specifically, the distribution $\widehat{\varphi}$ may not arise as the Fourier transform of a tempered distribution $\varphi$, yet we allow ourselves to refer to the operator $\Ha_{\varphi}^{\sigma,\tau}$.
	
	The rest of this subsection is devoted to the basic properties of the weight function $\omega_{\Omega}$.
	\begin{lm}\label{omega properties}
		
		Let $\Omega$ be an open and convex subset of $\mathbb{R}^{n}$ that does not contain lines. Then,
		\begin{enumerate}[font=\normalfont]
			\item $\omega_{\Omega}$ is continuous in $\mathbb{R}^{n}$ with $\supp\omega_{\Omega}=\overline{\Omega}.$
			\item For every bijective affine map $T:\mathbb{R}^{n}\to \mathbb{R}^{n}$, \[\omega_{T\Omega}(x)=|\det T|\omega_{\Omega}(T^{-1}(x)).\]
			\item The function $\omega_{\Omega}^{\frac{1}{n}}$ is concave in $\Omega$.
			\item The set $\omega_{\Omega}^{-1}([a,\infty))$ is convex for any $a>0.$
			\item For $x,y\in \Omega$, $$\omega_{\Omega}(\frac{x+y}{2})\geq 2^{-n}\max(\omega_{\Omega}(x),\omega_{\Omega}(y)).$$
			\item The sets $\omega_{\Omega}^{-1}(\{t\})$, $0<t<\sup\omega_{\Omega}(x)$ have measure zero.
		\end{enumerate}   
	\end{lm}
	\begin{proof} $\,$
		\begin{enumerate}
			\item The statement about the support follows from the fact that $\Omega$ is open and $\Omega\cap(2x-\Omega)\neq\emptyset$ if and only if $x\in \Omega$. Let $x_{n}\to x$. Then $\chi_{\Omega}(2x_{n}-y)\to \chi_{\Omega}(2x-y)$ for all $y\notin 2x-\partial\Omega$. Since every convex set has boundary of measure zero, continuity follows from for example the dominated convergence theorem.
			\item Suppose that $T = Ax + z$. From the change of variables $y=Tu$, we see that 
			\[ m(T\Omega\cap(2x-T\Omega)) = |\det A|\omega_{\Omega}(T^{-1}(x)). \]
			\item 
			Let $x,y\in \Omega$ and let $z_{1}\in\Omega\cap(2x-\Omega)$, $z_{2}\in\Omega\cap(2y-\Omega)$. Then $z_{1}=2x-\omega_{1}$ and $z_{2}=2y-\omega_{2}$, for some $\omega_{1},\omega_{2}\in\Omega$. Therefore $tz_{1}+(1-t)z_{2}=2(tx+(1-t)y)-(t\omega_{1}+(1-t)\omega_{2})$ and thus 
			\[t(\Omega\cap(2x-\Omega))+(1-t)(\Omega\cap(2y-\Omega))\subset \Omega\cap(2(tx+(1-t)y)-\Omega).\]
			By the Brunn-Minkowski theorem \cite[Theorem 3.13]{MR4180684}
			we know that the function $m(\cdot)^{\frac{1}{n}}$ is concave, and thus concavity follows by the monotonicity of the Lebesgue measure.
			\item 
			The concavity of $\omega_{\Omega}^{\frac{1}{n}}$ implies convexity of $\omega_{\Omega}^{-1}([a,\infty))$: for $x,y\in \omega_{\Omega}^{-1}([a,\infty))$ and $t\in (0,1)$, $$\omega_{\Omega}(tx+(1-t)y)\geq \left[t\omega_{\Omega}^{\frac{1}{n}}(x)+(1-t)\omega_{\Omega}^{\frac{1}{n}}(y)\right]^{n}\geq a.$$
			\item
			Furthermore,
			\begin{align*}
				\omega_{\Omega}(\dfrac{x+y}{2})&\geq \left[\frac{1}{2}\omega_{\Omega}^{\frac{1}{n}}(x)+\frac{1}{2}\omega_{\Omega}^{\frac{1}{n}}(y)\right]^{n}\geq \left[\frac{1}{2}\max(\omega_{\Omega}^{\frac{1}{n}}(x),\omega_{\Omega}^{\frac{1}{n}}(y))\right]^{n} \\ &=2^{-n}\max(\omega_{\Omega}(x),\omega_{\Omega}(y)).
			\end{align*}
			\item Since the boundary of a convex domain has measure zero, it suffices to show that $\partial\omega_{\Omega}^{-1}([t,\infty))=\omega_{\Omega}^{-1}(\{t\})$. The implication $"\subset"$ is immediate from continuity. For the converse implication let $\omega_{\Omega}(x)=t$. By concavity $x$ is not a local maximum nor a local minimum for $\omega_{\Omega}^{\frac{1}{n}}$ (in the case of concave or convex functions, local implies global), and thus we can find $x_n,y_n\in B(x,1/n)$ such that $\omega_{\Omega}^{\frac{1}{n}}(x_n)<\omega_{\Omega}^{\frac{1}{n}}(x)<\omega_{\Omega}^{\frac{1}{n}}(y_n)$. 
		\end{enumerate} 
	\end{proof}
	We single out the following consequence of property (5) of Lemma \ref{omega properties}.
	\begin{lm}\label{lower bound}
		Let $a > 1$. If $\omega_{\Omega}(x)=a^{j},\omega_{\Omega}(y)=a^{j'}$ and $\omega_{\Omega}(\frac{x+y}{2})=a^{k}$, $j,j',k \in \mathbb{Z}$, then $j,j'\leq k+L$ for some uniform constant $L$.
	\end{lm}
	\begin{proof}
		By Lemma \ref{omega properties}, 
		\begin{eqnarray}
			a^{k}&=&\omega_{\Omega}(\frac{x+y}{2})\geq 2^{-n}\max(\omega_{\Omega}(x),\omega_{\Omega}(y))= 2^{-n}\max(a^{j},a^{j'})\nonumber \\
			&=&\max(a^{j-n\log_{a}2},a^{j'-n\log_{a}2}) \nonumber
		\end{eqnarray}
		and thus $j,j'\leq k+n\log_{a}2.$
	\end{proof}
	
	\subsection*{Besov spaces adapted to $\Omega$}
	Roughly speaking, given $a > 1$, the definition of our Besov spaces is based on decompositions of the level sets 
	\[
	\Delta_{j}(\Omega)=\{x\in \Omega : \omega_{\Omega}(x)\in (a^{j},a^{j+1})\}, \qquad j\in\mathbb{Z}.
	\]
	Note that when $\Omega$ is bounded, the sets $\Delta_j(\Omega)$ are empty whenever $j > j_0$, for some $j_0$. In what follows, the reader should therefore simply assume that all such indices are smaller than $j_0$ in the bounded case.
	
	To define $B^s_{p,q}(\Omega)$, we assume that we can find \textit{parallelepipeds} $A^i_j$ such that
	\begin{enumerate}[label=(\roman*)]
		\item \label{1} There exists $M>0$ such that $A_{j}^{i}\subset \omega_{\Omega}^{-1}[a^{j-M},a^{j+M}]$.
		\item \label{2} There exist $c_1,c_2>0$ such that, $$c_1 a^j\leq m(A_j^i)\leq c_2 a^j.$$
		\item \label{3} There exists an $\epsilon\in (0,\frac{1}{2})$ such that 
		\[\Omega \subset \bigcup_{j,i} T^{-1}_{j,i}(-\frac{1}{2}+\epsilon,\frac{1}{2}-\epsilon)^{n},\]
		where 
		$T_{j,i}$ is the affine bijection that maps $A_{j}^{i}$ onto $(-\frac{1}{2},\frac{1}{2})^{n}$. 
		\item \label{4} There exists $C>0$ such that $$T_{j,i}A_{k}^{l}\subset C(-\frac{1}{2},\frac{1}{2})^{n}, \qquad \textrm{whenever } A_{j}^{i}\cap A_{k}^{l}\neq \emptyset.$$
		\item \label{5} For every choice $j,i,k,l$, the sets $J_{j,i,k,l}=\{(\beta,\gamma):(A_j^i+A_k^l)\cap 2A_\beta^\gamma\neq\emptyset\}$ have uniformly bounded cardinalities. As a special case we have that the sets $J_{j,i}=\{(\beta,\gamma):A_j^i\cap A_\beta^\gamma\neq\emptyset\}$ have uniformly bounded cardinalities.
	\end{enumerate}
	Whenever we refer to a decomposition $\{A_{j}^{i}\}$ for $\Omega$, we tacitly imply that properties \ref{1}-\ref{5} are satisfied. By an \textit{admissible set} we mean a convex open set $\Omega \subset \mathbb{R}^n$ that is free of lines and admits a decomposition $\{A_{j}^{i}\}$. 
	
	The conditions are mostly self-explanatory and natural; we only note that \ref{4} expresses that the orientation of the parallelepipeds varies "continuously" throughout the decomposition $\{A_{j}^{i}\}.$
	
	For an unbounded set $\Omega$, the index $i$ takes infinitely many values.
	For a bounded set $\Omega$, for given $j$, the index $i$ runs through $1$ to a finite number $n_j$, which roughly corresponds to the number of boxes $A^i_j$ needed to cover $\Delta_j(\Omega)$. In this latter case, a result of Schmuckenschl\"{a}ger \cite{MR1194970} gives precise estimates for $n_j\approx m(\Delta_j(\Omega))/a^j$ in two cases.
	\begin{theo}\label{nj}
		Let $\Omega$ be a bounded convex subset of $\mathbb{R}^{n}$. Then
		\begin{enumerate}
			\item If $\Omega$ is a polytope, $m(\Delta_{j}(\Omega))\approx |j|^{n-1}a^{j}$.
			\item If $\Omega$ has positive affine surface area, $m(\Delta_{j}(\Omega))\approx a^{\frac{2j}{n+1}}.$
		\end{enumerate}
	\end{theo}
	\begin{proof}
		When $\Omega$ is a polytope, \cite[Theorem 3]{MR1194970} says that $$\lim_{\delta\to 0}\frac{m(\Omega)-m(\omega_{\Omega}>\delta)}{\delta(\log\frac{1}{\delta})^{n-1}}=C,$$ where $C=C(\Omega)$ is finite and non-zero. For the other case, \cite[Theorem 2]{MR1194970} instead states that $$\lim_{\delta\to 0} \frac{m(\Omega)-m(\omega_{\Omega}>\delta)}{\delta^{\frac{2}{n+1}}}$$ exists and is non-zero. The asymptotics for $m(\Delta_j(\Omega))$ are immediate.
	\end{proof}
	
	Given $x \in \omega_\Omega^{-1}(\{t\})$, $t > 0$, let $H$ be a supporting hyperplane of $\omega_\Omega^{-1}(t,\infty)$ at $x$ with $\omega_\Omega^{-1}(t,\infty)\subset H^{-}$. For $c > 1$, we let $\Lambda_c(x)$ denote the convex cap\footnote{the notation will not cause issues in the event that the supporting hyperplane is not unique.}
	\begin{equation} \label{eq:convcap}
		\Lambda_{c}(x)=\omega_\Omega^{-1}\left(\frac{t}{c},t\right)\cap H^{+} = \omega_\Omega^{-1}\left(\frac{t}{c}, \infty\right)\cap H^{+}.
	\end{equation}
	The convex set $\Lambda_c(x)$ is always bounded. For if $A \subset \Lambda_c(x)$ is a parallelepiped with center $y$, then
	$$m(A) = m(A\cap (2y-A)) \leq m(\Omega\cap (2y-\Omega)) = \omega_{\Omega}(y) < \omega_{\Omega}(x) < \infty.$$
	
	The following theorem demonstrates that condition \ref{2} is natural. 	Because if condition \ref{1} holds, then by the hyperplane separation theorem, there is always an $x\in \omega_{\Omega}^{-1}(\{a^{j+M}\})$ such that $A_j^i \subset \Lambda_{a^{2M}}(x)$. On the other hand, for any such $x$, John's theorem \cite{John1948} ensures that we can find a parallelepiped $A \subset \Lambda_{a^{2M}}(x)\subset \omega_{\Omega}^{-1}[a^{j-M},a^{j+M}]$ of measure $m(A)\approx_n m(\Lambda_{a^{2M}}(x))$.	\begin{theo}\label{1and2}
		Let $\Omega$ be a convex subset of $\mathbb{R}^n$ not containing lines, and let $x \in \omega_\Omega^{-1}(\{t\})$, $t > 0$. Then there are $c_1,c_2>0$, depending only on $n$ and $c > 1$, such that
		$$ c_1\omega_\Omega(x)\leq m(\Lambda_c(x))\leq c_2\omega_\Omega(x).$$
	\end{theo}
	\begin{proof}
		Given $y\in (\Omega\cap H^{+})\setminus \Lambda_c(x)$ pick a point $z=\epsilon x+(1-\epsilon)y$, $0 \leq \epsilon < 1$, such that $\omega_{\Omega}(z) = \frac{c+1}{2c}t$. Then by the concavity of $\omega_{\Omega}^{\frac{1}{n}}$, we have that $\omega_{\Omega}(z)^{\frac{1}{n}}\geq \epsilon\omega_{\Omega}(x)^{\frac{1}{n}}$, and thus $\epsilon\leq (2c/(1+c))^{-\frac{1}{n}}$. Noting that $y = \frac{z}{1 - \varepsilon} - \frac{\varepsilon}{1 - \varepsilon}x$ and that the dilations $\lambda \Lambda_c(x) +(1-\lambda)x$ of the convex set $\Lambda_c(x)$, with respect to center $x$, are increasing in $\lambda \geq 1$, we conclude that
		$$\Omega\cap H^{+}\subset \frac{1}{1-(2c/(1+c))^{-\frac{1}{n}}}\Lambda_c(x)-\frac{(2c/(1+c))^{-\frac{1}{n}}}{1-(2c/(1+c))^{-\frac{1}{n}}}x.$$ 
		Thus $m(\Lambda_c(x))\approx_{c,n}m(\Omega\cap H^{+})$. 
		
		This immediately yields one side of the desired inequality. Since the Macbeath region $M(x)=\Omega\cap (2x-\Omega)$ is symmetric with respect to $x$, we have that
		$$\omega_{\Omega}(x) = 2m(M(x)\cap H^{+})\leq 2m(\Omega\cap H^{+}) \approx m(\Lambda_c(x)).$$ 
		For the converse inequality, we apply John's theorem \cite{John1948}: for every bounded convex set $C$ we can find an ellipse $P$ with center $y$ such that $P\subset C \subset y+n(P-y)$. Let $P$ be such an ellipse for $\Lambda_c(x)$. Then
		$$\omega_{\Omega}(x) > \omega_{\Omega}(y) \geq m(P\cap (2y-P))= m(P)\geq n^{-n} m(\Lambda_c(x)).$$
		This completes the proof.
	\end{proof}

	Now we define the Besov spaces adapted to an admissible set $\Omega$.
	
	\begin{defin}\label{Besov defin}
		Let $\Omega$ be an admissible subset of $\mathbb{R}^{n}$ equipped with the decomposition $\{A_{j}^{i}\}$. Let $\mathscr{P}(\{A_{j}^{i}\})$ be the collection of all families of smooth partitions of unity $(\varphi^{i}_{j})$ such that
		\begin{enumerate}
			\item $\displaystyle \supp\widehat{\varphi}^{i}_{j}\subset A^{i}_{j}$.
			\item $\displaystyle \sum_{j,i}\widehat{\varphi}^i_{j}=\chi_{\Omega}$.
			\item $\displaystyle \sup_{j,i} \|\varphi^i_j\|_{L^1(\mathbb{R}^n)} < \infty.$
		\end{enumerate}
		
		Let $(\varphi^i_j) \in \mathscr{P}(\{A_{j}^{i}\})$. For $1 \leq p,q \leq \infty$, and $s\in \mathbb{R}$, let 
		$$X_{p,q}^{s}(\Omega)=\{f\in S'(\mathbb{R}^{n}):\supp\widehat{f} \subset \Omega \text{ compact, and } a^{js}\|f\ast\varphi^{i}_{j}\|_{L^{p}}\in \ell^{q}\},$$
		endowed with the norm
		$$\|f\|_{B_{p,q}^{s}(\Omega)}=\left(\sum_{j,i}a^{qjs}\|f\ast\varphi^{i}_{j}\|^{q}_{L^{p}}\right)^{\frac{1}{q}},$$
		with the obvious modification when $q=\infty$. 
		We define the Besov space $B_{p,q}^{s}(\Omega)$ as the completion (weak*, if $q=\infty$) of $X_{p,q}^{s}$.
	\end{defin}
	We shall soon see that partitions of unity $(\varphi^{i}_{j}) \in \mathscr{P}(\{A_{j}^{i}\})$ exist and that $B_{p,q}^{s}(\Omega)$ is independent of the choice of partition of unity. However, the Besov space  $B_{p,q}^{s}(\Omega)$ may in principle depend on the decomposition $\{A_{j}^{i}\}$. Regardless, we choose to suppress the choice of partition from our notation; the characterizations in terms of Hankel operators will show that the Besov spaces in our main theorems are independent of the choice of decomposition.
	
	\begin{lm}\label{equiv of B}
		Let $\varphi\in S(\mathbb{R}^{n})$ be such that $\supp\widehat{\varphi}\subset (-\frac{1}{2},\frac{1}{2})^{n}$, and suppose that
		$$\|\partial^{\gamma}\widehat{\varphi}\|_{L^{\infty}}\leq C, \qquad \gamma\in\{0,1\}^{n}.$$ 
		Then $\|\varphi\|_{L^{1}}\leq C\sqrt{(2\pi)^{n}}$.
	\end{lm}
	\begin{proof}
		Let $p(x)=(1+x_{1}^{2})(1+x_{2}^{2})...(1+x_{n}^{2})=\sum_{\gamma\in\{0,1\}^{n}}x^{2\gamma}$. By Cauchy--Schwarz,
		\begin{align*}
			\|\varphi\|^{2}_{L^{1}} &\leq \pi^n \|p^{1/2}\varphi\|^{2}_{L^{2}}=\pi^{n}\sum_{\gamma\in\{0,1\}^{n}}\|x^{\gamma}\varphi\|_{L^{2}}^{2} \\
			&= \pi^{n}\sum_{\gamma\in\{0,1\}^{n}}\|\partial^{\gamma}\widehat{\varphi}\|_{L^{2}}^{2} \leq C^2(2\pi)^{n},
		\end{align*}
		as desired.
	\end{proof}
	\begin{prop}\label{exist}
		Let $\Omega$ be an admissible subset of $\mathbb{R}^{n}$ with decomposition $\{A_{j}^{i}\}$. Then $\mathscr{P}(\{A_{j}^{i}\}) \neq\emptyset$.
	\end{prop}
	\begin{proof}
		Let $\epsilon>0$ be as in property \ref{3}. Let $\varphi$ be a smooth function on $\mathbb{R}^{n}$, such that $0\leq \widehat{\varphi}\leq1$, $\widehat{\varphi}=1$ on $(-\frac{1-\epsilon}{2},\frac{1-\epsilon}{2})^{n}$, while $\supp \widehat{\varphi} \subset (-\frac{1}{2},\frac{1}{2})^{n}$. We may then define $\widehat{\varphi}_{j}^{i}=\widehat{\varphi}\circ T_{j,i}$.	Defining $\widehat{\psi}_{j}^{i}=\frac{\widehat{\varphi}_{j}^{i}}{\sum_{k,l}\widehat{\varphi}_{k}^{l}}$, our goal is to prove that $(\psi_{j}^{i})\in\mathscr{P}(\{A_{j}^{i}\})$. First we notice that properties \ref{3} and \ref{5} imply that $1\leq \sum_{k,l}\widehat{\varphi}_{k}^{l}\leq C$ in $\Omega$; therefore $\psi_{j}^{i}$ is well defined and smooth, with Fourier transform supported in $A_{j}^{i}$. Of course, $\sum_{j,i}\widehat{\psi}_{j}^{i}=\chi_{\Omega}$. Finally, property \ref{4} tells us that the operators $T_{k,l}T^{-1}_{j,i}$ have uniformly bounded partial derivatives whenever $A_{k}^{l}\cap A_{j}^{i}\neq\emptyset$.  Since 
		$$\widehat{\psi}_{j}^{i} \circ T_{j,i}^{-1} = \frac{\widehat{\varphi}}{\sum_{k,l} \widehat{\varphi} \circ T_{k,l} \circ T_{j,i}^{-1}}.$$
		the chain rule shows that there is a constant $C$ such that
		$$\|\partial^\gamma (\widehat{\psi}_{j}^{i} \circ T_{j,i}^{-1}) \|_{L^\infty} \leq C$$
		for every $\gamma \in \{0,1\}^n$, $j$, and $i$. Thus Lemma \ref{equiv of B} implies that
		\begin{equation*}
			\sup \|\psi_{j}^{i}\|_{L^{1}} = \sup \|\mathscr{F}^{-1}(\widehat{\psi}_{j}^{i}\circ T_{j,i}^{-1})\|_{L^{1}} < \infty. \qedhere
		\end{equation*}	 
	\end{proof}
	\begin{lm}\label{Besov properties}
		Let $\Omega$ be an admissible subset of $\mathbb{R}^{n}$ with decomposition $\{A_{j}^{i}\}$. Then 
		\begin{enumerate}
			\item
			The Besov spaces are well defined for $1 \leq p, q \leq \infty$ and $s \in \mathbb{R}$; different choices of partitions of unity in $\mathscr{P}(\{A_{j}^{i}\})$ yield equivalent norms.
			\item For $s,s'\in\mathbb{R}$, $1 \leq p,p',q,q' \leq \infty$, $\min(q, q') < \infty$, and $\theta\in(0,1)$,
			$$[B_{p,q}^{s},B_{p',q'}^{s'}]_{[\theta]}=B_{p_{\theta},q_{\theta}}^{s_{\theta}},$$ where $s_{\theta}=(1-\theta)s+\theta s'$, $\frac{1}{p_{\theta}}=\frac{1-\theta}{p}+\frac{\theta}{p'}$, $\frac{1}{q_{\theta}}=\frac{1-\theta}{q}+\frac{\theta}{q'}$, and $[\cdot,\cdot]_{[\theta]}$ refers to complex interpolation.
			\item If $p,q<\infty$ and $p',q'$ denote their respective conjugate indices, then  the dual space $(B_{p,q}^{s}(\Omega))^* \simeq B_{p',q'}^{-s}(\Omega)$.
		\end{enumerate}
	\end{lm}
	\begin{proof}
		We assume that $q < \infty$. Modifications of these arguments can be used to prove (1) and (2) when $q = \infty$.
		Let $(\varphi_{j}^{i}),(\psi_{j}^{i})\in \mathscr{P}(\{A_{j}^{i}\})$. To prove (1) it suffices, by symmetry, to prove norm inequality in one of the directions. Let 
		$$J_{j,i}:=\{(k,l):A_{j}^{i}\cap A_{k}^{l}\neq \emptyset\}.$$ Then, by considering Fourier supports, we have for $f\in S(\mathbb{R}^{n})$ that $$f\ast\varphi_{j}^{i}=\sum_{(k,l)\in J_{j,i}}f\ast\varphi_{j}^{i}\ast \psi_{k}^{l}.$$ Therefore 
		$$\|f\ast\varphi_{j}^{i}\|_{L^{p}}\leq \sum_{(k,l)\in J_{j,i}}\|f\ast \psi_{k}^{l}\|_{L^{p}}\|\varphi_{j}^{i}\|_{L^{1}}\leq \sup_{j,i}\|\varphi_{j}^{i}\|_{L^{1}}\sum_{(k,l)\in J_{j,i}}\|f\ast \psi_{k}^{l}\|_{L^{p}}.$$
		Since $\#J_{j,i} \leq C$ by property \ref{5}, and since the $\ell^{1}$- and $\ell^{q}$-norms are equivalent in finite dimensions, we get that 
		\begin{align*}
			\sum_{j,i}a^{qjs}\|f\ast\varphi^{i}_{j}\|^{q}_{L^{p}} &\leq  	K \sum_{j,i}a^{qjs}\big (\sum_{(k,l)\in J_{j,i}}\|f\ast \psi_{k}^{l}\|_{L^{p}}\big)^{q} \\
			&\leq K \sum_{j,i}\sum_{(k,l)\in J_{j,i}}a^{qjs}\|f\ast \psi_{k}^{l}\|_{L^{p}}^{q},
		\end{align*}
		By property \ref{1}, $|j-k|\leq 2M$ whenever $(k,l) \in J_{j,i}$, and thus
		$$	\sum_{j,i}a^{qjs}\|f\ast\varphi^{i}_{j}\|^{q}_{L^{p}} \leq K \sum_{k,l}\sum_{(j,i)\in J_{k,l}}a^{qks}\|f\ast \psi_{k}^{l}\|_{L^{p}}^{q} \leq  K \sum_{k,l}a^{qks}\|f\ast\psi^{l}_{k}\|^{q}_{L^{p}}.$$
		
		To prove (2), it suffices, by \cite[Theorem 6.4.2]{MR482275}, to prove that $B_{p,q}^{s}(\Omega)$ is a retract of $$\ell^{q}_{s}(L^p)=\{(f_{j,i}) \subset L^{p}:a^{sj}\|f_{j,i}\|_{L^{p}}\in\ell^{q}\}.$$
		Define $I:X_{p,q}^{s}(\Omega)\to \ell^{q}_{s}(L^p)$ and $P: \ell^{q}_{s}(L^p) \to X_{p,q}^{s}(\Omega)$  by 
		$$I(f)=(f\ast\varphi_{j}^{i}) \quad \text{and}\quad P(f_{j,i})=\sum_{j,i}(\sum_{(k,l)\in J_{j,i}}\varphi_{k}^{l})\ast f_{j,i},$$
		where $P$ is initially defined only for finitely supported sequences. 
		Then $PIf = f$ for $f \in X_{p,q}^{s}(\Omega)$. Therefore, we need only verify that the maps are bounded. $I$ is an isometry. For $P$ we observe that
		\begin{equation*}
			\|P(f_{j,i})\|^{q}_{B_{p,q}^{s}(\Omega)}
			=\sum_{\beta, \gamma}a^{sq\beta}\|\sum_{j,i}(\sum_{(k,l)\in J_{j,i}}\varphi_{k}^{l})\ast\varphi_{\beta}^{\gamma}\ast f_{j,i}\|_{L^{p}}^{q}. 
		\end{equation*}
		The sum of products $\sum_{(k,l)\in J_{j,i}}\widehat{\varphi}_{k}^{l}\widehat{\varphi}_{\beta}^{\gamma}$ is nonzero only if 
		$$(j,i)\in M_{\beta,\gamma}:=\{(j,i) \, : \, \textrm{there exists } (k,l)\in J_{j,i} \textrm{ such that   } (\beta,\gamma)\in J_{k,l}\}.$$
		Since $\#M_{\beta,\gamma} \leq C^2$, we get that
		\begin{eqnarray}
			\|P(f_{j,i})\|^{q}_{B_{p,q}^{s}(\Omega)}&=&\sum_{\beta,\gamma}a^{sq\beta}\|\sum_{(j,i)\in M_{\beta,\gamma}}(\sum_{(k,l)\in J_{j,i}}\varphi_{k}^{l})\ast\varphi_{\beta}^{\gamma}\ast f_{j,i}\|_{L^{p}}^{q} \nonumber \\
			&\leq&K\sum_{\beta,\gamma}\sum_{(j,i)\in M_{\beta,\gamma}}a^{sq\beta}\|(\sum_{(k,l)\in J_{j,i}}\varphi_{k}^{l})\ast\varphi_{\beta}^{\gamma}\ast f_{j,i}\|_{L^{p}}^{q} \nonumber \\
			&\leq&K\sum_{\beta,\gamma}\sum_{(j,i)\in M_{\beta,\gamma}}a^{sq\beta}\| f_{j,i}\|_{L^{p}}^{q} \nonumber \\
			&\leq&K\sum_{j,i}a^{sqj}\| f_{j,i}\|_{L^{p}}^{q}. \nonumber 
		\end{eqnarray}
		In the last inequality we used that $(j,i)\in M_{\beta,\gamma}$ if and only if $(\beta,\gamma)\in M_{j,i}$ and if so, property \ref{1} yields $|\beta-j|\leq 4M$.
		
		To prove (3) note first that any $g \in X_{p',q'}^{-s}(\Omega)$ induces a continuous functional on $f\in X_{p,q}^{s}(\Omega)$, since
		\begin{eqnarray}
			|\langle f,g\rangle |&=&\big|\sum_{j, i}\sum_{(k,l)\in J_{j, i}}\langle f\ast \varphi_{j}^{i},g\ast \varphi_{k}^{l}\rangle \big|\leq \sum_{j, i}\sum_{(k,l)\in J_{j, i}}\| f\ast \varphi_{j}^{i}\|_{L^{p}} \|g\ast \varphi_{k}^{l}\|_{L^{p'}} \nonumber  \\
			&\leq& \left(\sum_{j, i}\sum_{(k,l)\in J_{j, i}}a^{sjq}\| f\ast \varphi_{j}^{i}\|_{L^{p}}^{q} \right)^{\frac{1}{q}}\left(\sum_{k, l}\sum_{(j, i)\in J_{k, l}} a^{-jq's}\|g\ast \varphi_{k}^{l}\|_{L^{p'}}^{q'} \right)^{\frac{1}{q'}} \nonumber \\
			&\leq& K\|f\|_{B_{p,q}^{s}(\Omega)}\|g\|_{B_{p',q'}^{-s}(\Omega)}, \nonumber
		\end{eqnarray}
		where $\langle \cdot, \cdot \rangle$ denotes the $L^2$-pairing. It follows that $B_{p',q'}^{-s}(\Omega) \subset (B_{p,q}^{s}(\Omega))^*$. Conversely, suppose that $\nu \in (B_{p,q}^{s}(\Omega))^*$. By Hahn-Banach, we can find an extension $(g_{j,i}) \in \ell^{q'}_{-s}(L^{p'})$ with the same norm. Let $g_m = \sum_{-m\leq j\leq m}\sum_{i} g_{j, i} \ast \varphi_{j}^i \in X_{p',q'}^{-s}(\Omega)$. Then, for any $f\in X_{p,q}^{s}(\Omega)$, we have that 
		$$\langle g_m, f\rangle = \sum_{-m\leq j\leq m}\sum_{i}\langle g_{j, i}\ast\varphi_{j}^{i},f\rangle = \sum_{-m\leq j\leq m}\sum_{i}\langle g_{j, i},f\ast\varphi_{j}^{i}\rangle = \nu(f)$$
		for all sufficiently large $m$ (depending on $f$). Furthermore, 
		\begin{align*}
			\|g_m\|_{B_{p',q'}^{-s}(\Omega)}^{q'} &= \sum_{j, i}a^{-jsq'}\|\sum_{(k,l)\in J_{j,i},-m\leq k\leq m}g_{k,l}\ast\varphi_{k}^{l}\ast\varphi_{j}^{i}\|_{L^{p'}}^{q'} \nonumber \\
			&\leq K \sum_{j, i}\sum_{(k,l)\in J_{j,i}}a^{-jsq'}\|g_{k,l}\|_{L^{p'}}^{q'} = K\sum_{k, l}\sum_{(j, i) \in J_{k,l}}a^{-jsq'}\|g_{k, l}\|_{L^{p'}}^{q'} \nonumber \\
			&\leq K\sum_{k, l} a^{-ksq'}\|g_{k, l}\|_{L^{p'}}^{q'}  = K \|\nu\|_{(B_{p,q}^{s}(\Omega))^*}^{q'}.
		\end{align*}
		Therefore the series $g = \lim_{m \to \infty} g_m$ converges absolutely (weak*, if $q = 1$) in $B_{p',q'}^{-s}(\Omega)$ to an element $g$ which, as a functional on $B_{p,q}^{s}(\Omega)$, equals $\nu$.
	\end{proof}
	
	\section{Sufficient conditions for trace class and boundedness} \label{sec:suf}
	In this section we will prove natural sufficient conditions for an extended Hankel operator $\Ha_\varphi^{\sigma, \tau}$ to be either trace class or Hilbert-Schmidt. We therefore obtain a criterion for membership in the Schatten class $S^p$, for $1 \leq p \leq 2$. We shall also give a useful sufficient criterion for certain extended Hankel operators to be bounded, which we shall apply when considering smooth domains with positive curvature, as well as polytopes.
	
	In what follows, we shall identify any kernel $k(x,y)$ on $\Omega \times\Omega$ with an integral operator on $L^2(\Omega)$, and we shall write $\|k\|_{S^p}$ or $\|k(x,y)\|_{S^p}$ for its $p$-Schatten norm, $1 \leq p \leq \infty$. In this notation, note that any $m(x,y)$ belonging to the projective tensor product $L^\infty(\Omega) \hat{\otimes} L^\infty(\Omega)$ is a Schur multiplier:
	\begin{equation} \label{eq:schurinf}
		\|mk\|_{S^p} \leq \|m\|_{L^\infty(\Omega) \hat{\otimes} L^\infty(\Omega)} \|k\|_{S^p}.
	\end{equation}
	See \cite[Section 3]{MR924766} for further details. By integration we obtain the following useful lemma.
	\begin{lm} \label{lem:schur}
		If $\varphi \in L^1$, then
		$$\|\hat{\varphi}(x+y)k(x,y)\|_{S^p} \leq \|\varphi\|_{L^1} \|k\|_{S^p}, \qquad 1 \leq p \leq \infty.$$
	\end{lm} 
	\begin{proof}
		This follows from \eqref{eq:schurinf}, since $\hat{\varphi}(x+y) = \int_{\mathbb{R}^n} e^{-2\pi i x \xi}e^{-2\pi i y \xi} \, \varphi(\xi) \, d\xi$ is an absolutely convergent integral of elements from $L^\infty(\Omega) \hat{\otimes} L^\infty(\Omega)$.
	\end{proof}
	The sets 
	\begin{equation} \label{eq:gijdef}
		G_{j,i} := \bigcup_{x\in A_j^i}\Omega\cap (2x-\Omega) = \bigcup_{x\in A_j^i} M(x),
	\end{equation}
	play a central role in the sufficiency proofs. Recall that $M(x) = \Omega\cap (2x-\Omega)$ is the Macbeath region centered at $x$, and note that $x \in G_{j,i}$ if and only if $x \in \Omega$ and there is $y \in \Omega$ such that $\frac{x+y}{2} \in A^i_j$.
	We begin by establishing some basic properties of these sets.
	\begin{lm}\label{Gestim}
		For any indices $j, i$, the set $G_{j,i}$ is convex and there is $C>0$, depending only on $\Omega$, such that $\omega_{\Omega}(x)\leq C a^j$ for $x\in G_{j,i}$.
	\end{lm}
	\begin{proof}
		For convexity let $x_1,x_2\in G_{j,i}$ and $\lambda \in (0,1)$. Then by definition we can find $y_1,y_2\in \Omega$ such that $\frac{x_1+y_1}{2},\frac{x_2+y_2}{2}\in A_j^i$. By convexity of $A_j^i$, the vectors $\lambda x_1+(1-\lambda )x_2$ and $\lambda y_1+(1-\lambda) y_2$ satisfy $\frac{\lambda x_1+(1-\lambda )x_2+\lambda y_1+(1-\lambda) y_2}{2}=\lambda \frac{x_1+y_1}{2}+(1-\lambda) \frac{x_2+y_2}{2}\in A_j^i$. Thus $\lambda x_1+(1-\lambda )x_2 \in G_{j,i}$.
		
		For the estimate, suppose that $x\in G_{j,i}$ and $y\in \Omega$ with $\frac{x+y}{2}\in A_j^i$. Then, by the concavity of $\omega_{\Omega}^{\frac{1}{n}}$, 
		\[\omega_{\Omega}^{\frac{1}{n}}(x)\leq 2 \omega_{\Omega}^{\frac{1}{n}}(\frac{x+y}{2})-\omega_{\Omega}^{\frac{1}{n}}(y)\leq 2 \omega_{\Omega}^{\frac{1}{n}}(\frac{x+y}{2})\lesssim a^{\frac{j}{n}}. \qedhere \]
	\end{proof}
	The following lemma, inspired by \cite[Lemma 6.2]{MR2327289}, is key to understanding the behavior of $G_{j,i}$. We let $M(x,\lambda)=(1-\lambda)x+\lambda M(x)=x+\lambda(x-\Omega)\cap (\Omega-x)$ denote the $\lambda$-dilation of $M(x)$, with respect to the center $x$.
	\begin{lm}\label{Barany}
		Let $x, y \in \Omega$ and $\epsilon\in (0,1)$. If $M(x,\epsilon)\cap M(y,1-\epsilon)\neq \emptyset$, then $M(y)\subset M(x,\frac{2+\epsilon}{\epsilon})$.
	\end{lm}
	\begin{proof}
		Let us write $\rho=\frac{2+\epsilon}{\epsilon}$, and let $b\in M(y)$. By assumption, there exists an element $a\in M(x,\epsilon)\cap M(y,1-\epsilon)$. Accordingly, there are $k_1,k_2\in \Omega$ such that $a=x+\epsilon(x-k_1)=y+(1-\epsilon)(k_2-y)$. Since $b \in M(y)$, there furthermore exists $k_3\in\Omega$ such that
		\begin{eqnarray}
			b&=&2y-k_3=\frac{2}{\epsilon}(x+\epsilon(x-k_1)-(1-\epsilon)k_2)-k_3  \nonumber \\
			&=&x+\rho\left(x-\frac{2}{\rho}k_1-\frac{2(1-\epsilon)}{\epsilon\rho}k_2-\frac{1}{\rho}k_3\right) \nonumber.
		\end{eqnarray}
		Thus $b \in x+\rho (x-\Omega)$, since $\frac{2}{\rho}+\frac{2(1-\epsilon)}{\epsilon\rho}+\frac{1}{\rho}=1.$ On the other hand,  all of $\Omega$ is contained in the dilation $x+\rho(\Omega-x)$, since $\rho > 1$. That is, $b \in M(x, \rho).$
	\end{proof}
	We can now precisely describe the behavior of $m(G_{j,i})$.
	\begin{lm}\label{Gij}
		There exist $C_1,C_2>0$, depending only on $\Omega$, such that $$C_1 a^j\leq m(G_{j,i})\leq C_2a^j,$$
		for all indices $j,i$. The set $G_{j,i}$ was defined in \eqref{eq:gijdef}.
	\end{lm}
	\begin{proof}
		The first inequality is obvious, since $A_j^i\subset G_{j,i}$. 
		
		For the other inequality, let $x_{j,i}$ be the center of $A_j^i$. Then, 
		$$M(x_{j,i},\epsilon)\supset (1-\epsilon)x_{j,i}+\epsilon (A_j^i \cap (2x_{j,i}-A_j^i))= (1-\epsilon)x_{j,i}+\epsilon A_j^i.$$ 
		By property \ref{3}, there is thus $\epsilon\in (0,\frac{1}{2})$ such that the contractions
		$$2\epsilon x_{j,i}+(1-2\epsilon) A_j^i = T^{-1}_{j,i}(-\frac{1}{2}+\epsilon,\frac{1}{2}-\epsilon)^n$$
		cover $\Omega$. For fixed $j,i$, we therefore have that for any $y \in A_j^i$, there is $k, l$ such that 
		$$y \in 2\epsilon x_{k,l}+(1-2\epsilon) A_k^l \subset A_k^l \cap  M(x_{k,l},1-2\epsilon).$$
		Thus, by Lemma \ref{Barany} and property \ref{5}, we conclude that
		$$ m(G_{j,i})= m\left(\bigcup_{y\in A_j^i}M(y) \right)\leq m\left (\bigcup_{k,l\in J_{j,i}} M\left(x_{k,l},\frac{3-2\epsilon}{1-2\epsilon}\right)\right)\leq C a^j.$$
		This finishes the proof.
	\end{proof}
	The following estimates are key to our sufficiency criteria for $S^1$ and $S^2$.
	\begin{theo}\label{Gjilevel}
		Let $\Omega$ be a convex subset of $\mathbb{R}^n$ that does not contain lines. Then there is $C=C(\Omega,n)$ such that 
		$$m(G_{j,i}\cap \Delta_k(\Omega))\leq Ca^{\frac{2k+j(n-1)}{n+1}},$$ for all $j,i,k$.
		Furthermore, if $\Omega$ is a polyhedron, the better estimate
		$$m(G_{j,i}\cap \Delta_k(\Omega))\leq C(1 + |k-j|^{n-1}) a^{k},$$ holds.
	\end{theo}
	\begin{proof}
		First, let us observe the following interaction between normalization and the level sets $\Delta_k$. Let $K$ be a bounded convex set and let $F_{K}(x)=m(K)^{-\frac{1}{n}}x$. Then,
		\begin{align*} 
			\Delta_s(m(K)^{-\frac{1}{n}}K)& =\{x\in m(K)^{-\frac{1}{n}}K:\omega_{m(K)^{-\frac{1}{n}}K}(x)\in (a^{s-1},a^s)\} \\
			&=F_{K}(\{y\in K:m(K)^{-1}\omega_{K}(y)\in (a^{s-1},a^s)\}) \\
			&= F_K(\Delta_{s+\log_a m(K)}(K)).
		\end{align*}
		
		In the general case, we rely on the estimate
		\begin{equation} \label{eq:floating}
			m(\{x\in K:\omega_K(x)<t\})\leq C t^{\frac{2}{n+1}},
		\end{equation}
		which holds for all bounded convex sets $K$ with $m(K) = 1$, for a constant $C$ which only depends on the dimension $n$, see \cite[Theorem~1]{MR986636}. Actually, \cite{MR986636} sees $\omega_K$ replaced by a volume induced by the convex minimal cap of $K$ at $x$; however, this is an equivalent quantity, see \cite[Lemmas~6.3 and 6.5]{MR2327291}. Applied with $K = m(G_{j,i})^{-\frac{1}{n}}G_{j,i}$, and using Lemma \ref{Gij}, we get that
		$$ m(\Delta_{s}(G_{j,i}))=m(F_{G_{j,i}}^{-1}(\Delta_{s-\log_a m(G_{j,i})}(m(G_{j,i})^{-\frac{1}{n}}G_{j,i})))\leq C(\Omega,n) a^ja^{(s-j)\frac{2}{n+1}}.$$
		The result now follows by observing that $\omega_{G_{j,i}}\leq \omega_{\Omega}$, and thus
		$$G_{j,i}\cap \Delta_k(\Omega)\subset \bigcup_{s\leq k}\Delta_s(G_{j,i}).$$
		
		When $\Omega$ is a polyhedron, a better estimate than \eqref{eq:floating} is available. We start by observing that, since the Minkowski sum of polyhedra is a polyhedron, the set $G_{j,i}=\Omega\cap (2A_j^i-\Omega)$ is a bounded polyhedron, that is, a polytope. For a polytope $P$ of measure $1$ with extreme points $\ext(P)$, \cite[Lemma 1.8]{MR1119928} gives the estimate 
		$$m(\{x\in P:\omega_{P}(x)<t\})\leq C(n)\varphi(P)t|\log\frac{1}{t}|^{n-1},$$ where $\varphi(P)$ is defined as follows:
		\begin{enumerate}
			\item if $P$ is a polyhedron of dimension $1$ (an interval), then $\varphi(P)=2$;
			\item if $P$ is a polyhedron of dimension $n$, then $\varphi(P)=\sum_{x \in \ext(P)}\varphi(P\cap H_x)$, where $H_x$ is a hyperplane that separates $x$ from the other extreme vectors of $P$. 
		\end{enumerate}
		Geometrically, $\varphi(P)$ represents the number of complete flags of $P$, that is, the number of sequences $F_0\subset F_1\subset ... \subset F_{n-1}\subset P$, where $F_i$ are faces of $P$ of dimension $i$. To finish the proof, it therefore suffices to observe that $\varphi(m(G_{j,i})^{-\frac{1}{n}}G_{j,i})=\varphi(G_{j,i})\leq C(\Omega)$, which is evident from the elementary properties $\varphi(A\cap B)\leq \varphi(A)\varphi(B)$ and $\varphi(A+B)\leq \varphi(A)\varphi(B)$.
		
	\end{proof}
	We now present the promised trace class condition. The proof can be understood as a far-reaching generalization of Peller's original approach \cite[Ch 6.1]{MR1949210} to the trace class condition for classical Hankel operators.
	\begin{theo}\label{S1}
		Let $\Omega$ be an admissible subset of $\mathbb{R}^{n}$ with decomposition $\{A_{j}^{i}\}$, and equip $2\Omega$ with the decomposition $\{2A_{j}^{i}\}$. Suppose that $\sigma, \tau>-\frac{1}{n+1}$ and that $\varphi\in B_{1,1}^{\sigma+\tau+1}(2\Omega)$. Then $\Ha_{\varphi}^{\sigma,\tau}\in S^{1}$ and there is a constant $C = C(\Omega, \sigma, \tau) > 0$ such that $\|\Ha_{\varphi}^{\sigma,\tau}\|_{S^{1}}\leq C \|\varphi\|_{B_{1,1}^{\sigma+\tau+1}(2\Omega)}.$ If $\Omega$ is a polyhedron the same result holds for all $\sigma,\tau>-\frac{1}{2}$.
	\end{theo}
	\begin{proof}
		Let $(\varphi^{i}_{j}) \in \mathscr{P}(\{2A_{j}^{i}\})$. By Lemma \ref{lem:schur}, we then have that 
		\begin{eqnarray}\|\Ha_\varphi^{\sigma,\tau}\|_{S^1}&=&\|\widehat{\varphi}(x+y)\omega_{\Omega}^\sigma(x)\omega_{\Omega}^\tau(y)\chi_{\Omega}(x)\chi_{\Omega}(y)\|_{S^1} \nonumber  \\
			&\leq&\sum_{j,i}\|\widehat{\varphi}(x+y)\widehat{\varphi}_j^i(x+y)\omega_{\Omega}^\sigma(x)\omega_{\Omega}^\tau(y)\chi_{G_{j,i}}(x)\chi_{G_{j,i}}(y)\|_{S^1} \nonumber \\
			&\leq&\sum_{j,i}\|\varphi\ast\varphi_j^i\|_{L^1}\|\omega_{\Omega}^\sigma(x)\omega_{\Omega}^\tau(y)\chi_{G_{j,i}}(x)\chi_{G_{j,i}}(y)\|_{S^1} \nonumber \\
			&\leq&\sum_{j,i}\|\varphi\ast\varphi_j^i\|_{L^1}\|\omega_{\Omega}^\sigma(x)\chi_{G_{j,i}}(x)\|_{L^2}\|\omega_{\Omega}^\tau(y)\chi_{G_{j,i}}(y)\|_{L^2}. \nonumber
		\end{eqnarray}
		By Lemma \ref{Gestim} and Theorem \ref{Gjilevel} there is $C = C(\Omega) > 0$ such that
		\begin{align*}
			\int_{G_{j,i}}\omega_{\Omega}^{2\sigma}(x)dx &\approx\sum_{k<j+C}m(G_{j,i}\cap \Delta_k(\Omega))a^{2k\sigma} \\ &\lesssim a^{j\frac{n-1}{n+1}}\sum_{k<j+C}a^{2k(\sigma+\frac{1}{n+1})}\approx a^{j(2\sigma+1)},
		\end{align*} 
		whenever $\sigma>-\frac{1}{n+1}$. Thus
		\begin{equation}\label{eq:S1ineq}
			\|\Ha_\varphi^{\sigma,\tau}\|_{S^1} \lesssim \sum_{j,i}a^{j(\sigma + \tau + 1)} \|\varphi\ast\varphi_j^i\|_{L^1}.
		\end{equation}
		When $\Omega$ is a polyhedron, we can apply the better estimate of Theorem \ref{Gjilevel} in order to obtain the same inequality \eqref{eq:S1ineq}, but this time for all $\sigma, \tau >-\frac{1}{2}$. 
	\end{proof}
	We can apply a similar argument to obtain the sufficiency part of the Hilbert--Schmidt characterization. 
	\begin{theo}\label{S2}
		Let $\Omega$ be an admissible convex subset of $\mathbb{R}^{n}$ with decomposition $\{A_{j}^{i}\}$, and equip $2\Omega$ with the decomposition $\{2A_{j}^{i}\}$. Suppose that $\sigma,\tau>-\frac{1}{n+1}$, $\sigma+\tau>-\frac{1}{n+1}$, and that $\varphi\in B_{2,2}^{\sigma+\tau+\frac{1}{2}}(2\Omega)$. Then $\Ha_{\varphi}^{\sigma,\tau}\in S^{2}$ and there is a constant $C = C(\Omega, \sigma, \tau) >0$ such that $\|\Ha_{\varphi}^{\sigma,\tau}\|_{S^{2}}\leq C \|\varphi\|_{B_{2,2}^{\sigma+\tau+\frac{1}{2}}(2\Omega)}$. If $\Omega$ is a polyhedron the same result holds for $\sigma,\tau>-\frac12\text{ and }\sigma+\tau>-\frac12$.
	\end{theo}
	\begin{proof}
		It is equivalent to prove that
		$$\omega_{\Omega}^{2\sigma} \ast \omega_{\Omega}^{2\tau} (x)\lesssim \omega_{\Omega}^{2\sigma+2\tau+1}(x/2), \qquad x \in 2\Omega,$$
		since 
		$$\|\Ha_\phi^{\sigma,\tau}\|_{S^2}^2 = \int_{2\Omega} |\widehat{\phi}(x)|^2 \omega_{\Omega}^{2\sigma} \ast \omega_{\Omega}^{2\tau} (x) \, dx.$$
		Let $\Omega$ be a polyhedron with $N$ facets, and set $G=M_\Omega(x/2)$. From Theorem \ref{Gjilevel}, we know that
		\begin{equation}\label{polyestim}
			\int_G \omega_G^{-r}(y)\,dy\lesssim_{n,N,r} m(G)^{1-r},\qquad 0<r<1.
		\end{equation}
		and, furthermore, by Lemma \ref{omega properties}, $$\max(\omega_\Omega(y),\omega_\Omega(x-y))\leq 2^n\omega_\Omega(x/2),\quad y\in G.$$
		Therefore, if $\sigma,\tau\geq 0$, we have that
		$$\omega_\Omega^{2\sigma}\ast \omega_\Omega^{2\tau}(x)=\int_{G}\omega_\Omega^{2\sigma}(x-y) \omega_\Omega^{2\tau}(y)dy\lesssim \omega_\Omega^{2\sigma+2\tau} (x/2)m(G)\approx \omega_\Omega^{2\sigma+2\tau+1} (x/2) .$$
        
		When $\sigma<0\leq \tau$, we use that
		$$\omega_\Omega^{2\sigma}(x-y)\leq \omega_G^{2\sigma}(x-y)=\omega_G^{2\sigma}(y).$$
		Since $\sigma > -1/2$, we can apply \eqref{polyestim}, yielding that
		\begin{align*} 
			\omega_\Omega^{2\sigma}\ast \omega_\Omega^{2\tau}(x)&=\int_{G}\omega_\Omega^{2\sigma}(x-y) \omega_\Omega^{2\tau}(y)dy\lesssim \omega_\Omega^{2\tau}(x/2)\int_{G}\omega_\Omega^{2\sigma}(x-y) dy  \\
			& \lesssim_{n,N,\sigma}\omega_\Omega^{2\tau}(x/2)m(G)^{1+2\sigma} \approx \omega_\Omega^{2\sigma+2\tau+1} (x/2). 
		\end{align*} 
        The case when $\tau < 0 \leq \sigma$ is symmetric. When $\sigma,\tau<0$ we argue in the same way, using that
		$$\omega_\Omega^{2\sigma}(x-y)\leq\omega_G^{2\sigma}(y)\quad \text{and} \quad \omega_\Omega^{2\tau}(y)\leq\omega_G^{2\tau}(y),$$
        and recalling that $\sigma+\tau>-1/2$ in order to apply \eqref{polyestim}.
        
        For a general set $\Omega$, the proof is the same, but equation \eqref{polyestim} should be replaced by
		$$\int_G \omega_G^{-r}(y)\,dy\lesssim_{n,r} m(G)^{1-r},\qquad 0<r<\frac{2}{n+1},$$
        in accordance with Theorem \ref{Gjilevel}.
	\end{proof}
	Interpolating between Theorem \ref{S1} and \ref{S2} we get the following corollary, which is Theorem~\ref{thm:mainsuf}.
	\begin{cor}\label{1 to 2 suf}
		Let $\Omega$ be an admissible subset of $\mathbb{R}^{n}$ with decomposition $\{A_{j}^{i}\}$, equip $2\Omega$ with the decomposition $\{2A_{j}^{i}\}$, and let $1\leq p\leq 2$ and $\sigma,\tau>-\frac{1}{n+1}$ with $\sigma+\tau>-\frac{2}{p(n+1)}$. Suppose that $\varphi\in B_{p,p}^{\sigma+\tau+\frac{1}{p}}(2\Omega)$. Then $\Ha_{\varphi}^{\sigma,\tau}\in S^{p}(\PW(\Omega))$, and there is a constant $C = C(\Omega, \sigma, \tau, p) > 0$ such that 
		$$\|\Ha_{\varphi}^{\sigma,\tau}\|_{S^{p}}\leq C\|\varphi\|_{B_{p,p}^{\sigma+\tau+\frac{1}{p}}(2\Omega)}.$$ If $\Omega$ is a polyhedron the same is true for $\sigma,\tau>-\frac12$ with $\sigma+\tau>-\frac1p$.
	\end{cor}
	\begin{proof}
		In light of Lemma~\ref{Besov properties}, we can apply the method of complex interpolation. We may without issue consider the operator $\Ha_{\varphi}^{\sigma,\tau}$ to be defined for all complex $\sigma,\tau\in\mathbb{C}$. Note then that $\Ha_{\varphi}^{\sigma,\tau}$ can be obtained by composition of $\Ha_{\varphi}^{\Real\sigma,\Real\tau}$ with two unitary factors (corresponding to multiplication by $\omega_{\Omega}^{i\Imaginary\sigma}(x)$ and $\omega_{\Omega}^{i\Imaginary\tau}(y)$ on the Fourier side). Therefore $$\|\Ha_{\varphi}^{\sigma,\tau}\|_{S^p}=\|\Ha_{\varphi}^{\Real\sigma,\Real\tau}\|_{S^p}.$$
		
		For given $\sigma_1,\sigma_2,\tau_1,\tau_2 \in \mathbb{R}$, consider in the strip $0 \leq \Real z \leq 1$, the functions
		$$\sigma(z)=(1-z)\sigma_{1}+z\sigma_{2}, \; \tau(z)=(1-z)\tau_{1}+z\tau_{2}, \; \frac{1}{\rho(z)} = 1 - \frac{\Real z}{2}.$$
		By Theorems \ref{S1} and \ref{S2} we have for appropriate $\sigma_j$ and $\tau_j$ that
		$$\|\Ha^{\sigma_j, \tau_j}_\varphi\|_{S^{j}} \leq C \|\varphi\|_{B_{j,j}^{\sigma_j+\tau_j+\frac{1}{j}}(2\Omega)}, \qquad j=1,2,$$
		and complex interpolation therefore yields 
		$$\|\Ha^{\sigma(z), \tau(z)}_\varphi\|_{S^{\rho(z)}} \leq C  \|\varphi\|_{B_{\rho(z),\rho(z)}^{\Real(\sigma(z)+\tau(z))+\frac{1}{\rho(z)}}(2\Omega)}.$$
		In order to obtain the stated result, we require that $\rho(z) = p$, and we therefore choose to consider $z = 2 (1 - 1/p)$. The proof is then finished by picking $\sigma_1, \sigma_2, \tau_1, \tau_2$ such that
		\begin{equation*}
			\sigma=(1-z)\sigma_1+z\sigma_{2} \textrm{ and } \tau=(1-z)\tau_1+z\tau_{2}.  
		\end{equation*}
		The constraints give $\sigma,\tau>-\frac1{n+1}$ and $\sigma+\tau>-\frac{2}{p(n+1)}$; for polyhedra they give $\sigma,\tau>-\frac12$ and $\sigma+\tau>-\frac1p$.
	\end{proof}
	
	To finish this subsection, we will present a sufficient condition for the boundedness of $\Ha_{\varphi}^{\sigma,\tau}$. Given $j,i,k,l$, let 
	$\beta = \beta(j, i, k, l)$ be an index such that $(A_{j}^{i}+A_{k}^{l})\cap 2A_{\beta}^{\gamma}\neq\emptyset$, that is, $(\beta, \gamma) \in J_{j,i,k,l}$. There is only a finite number of such $\beta$ to choose from; recall that property \ref{5} says that $\# J_{j,i,k,l} \leq C$. For $\sigma, \tau, \rho \in \mathbb{R}$, we introduce the matrix
	\begin{equation} \label{eq:Tsigmataurho}
		T_{\sigma, \tau, \rho} = (a^{j\sigma}a^{k\tau}a^{-\beta\rho})_{(j,i),(k,l)},
	\end{equation}
	understood as an operator $T_{\sigma, \tau, \rho} \colon \ell^2(\mathbb{Z}^{2}) \to \ell^2(\mathbb{Z}^{2})$ (we understand the $(j,i,k,l)$:th element of $T_{\sigma, \tau, \rho}$ to be $0$ if either $A_j^i$ or $A_k^l$ are not part of the decomposition $\{A_j^i\}$). Note that property \ref{1} shows that the boundedness of $T_{\sigma, \tau, \rho}$ is independent of the choice of $\beta$-function.
	\begin{lm}\label{Sinfty}
		Let $\Omega$ be an admissible subset of $\mathbb{R}^{n}$ with decomposition $\{A_{j}^{i}\}$, and equip $2\Omega$ with the decomposition $\{2A_{j}^{i}\}$. Let $\sigma,\tau,\rho\in\mathbb{R}$ be such that the operator $T_{\sigma,\tau,\rho}$ is bounded.
		Then $\varphi\in B_{\infty,\infty}^{\rho}(2\Omega)$ implies that $\Ha_{\varphi}^{\sigma,\tau}$ is bounded, and there is a constant $C > 0$ such that
		\begin{equation*}
			\|\Ha_{\varphi}^{\sigma,\tau}\| \leq C \|T_{\sigma,\tau,\rho}\|\|\varphi\|_{B_{\infty,\infty}^{\rho}(2\Omega)}.
		\end{equation*}
	\end{lm}
	\begin{proof}
		Choose a measurable partition $\{E_j^i\}$ of $\Omega$ with $E_j^i\subset A_j^i$. Then
		$$\widehat{\Ha_{\varphi}^{\sigma,\tau}f}(y) = \sum_{j,i,k,l} I_{j,i,k,l} \hat{f}(y), \qquad y \in \Omega,$$
		where the integral operators $I_{j, i, k, l} \colon L^2(\Omega) \to L^2(\Omega)$ have kernel
		$$I_{j,i,k,l}(x,y) = \widehat{\varphi}(x+y) \omega_{\Omega}^{\sigma}(x)\chi_{E_{j}^{i}}(x)\omega_{\Omega}^{\tau}(y)\chi_{E_{k}^{l}}(y).$$
		By \eqref{eq:schurinf},
		$$	\|I_{j,i,k,l}\| \leq K a^{j\sigma}a^{k\tau}\|\widehat{\varphi}(x+y)\chi_{E_{j}^{i}}(x)\chi_{E_{k}^{l}}(y)\|.$$
		Let $(\varphi_{j}^i) \in\mathscr{P}(\{2A_{j}^{i}\})$. Then 
		$$\chi_{E_{j}^{i}}(x)\chi_{E_{k}^{l}}(y)=\sum_{(\beta,\gamma)\in J_{j,i,k,l}}\widehat{\varphi}_{\beta}^{\gamma}(x+y)\chi_{E_{j}^{i}}(x)\chi_{E_{k}^{l}}(y),$$ 
		and thus
		\begin{align*}
			\|I_{j,i,k,l}\|&\leq 
			Ka^{j\sigma}a^{k\tau}\sum_{(\beta,\gamma)\in J_{j,i,k,l}}\|\widehat{\varphi}(x+y)\widehat{\varphi}_{\beta}^{\gamma}(x+y)\| \\ 
			&= Ka^{j\sigma}a^{k\tau}\sum_{(\beta,\gamma)\in J_{j,i,k, l}}\|\varphi\ast\varphi_{\beta}^{\gamma}\|_{L^{\infty}} \nonumber \\
			&\leq K\|\varphi\|_{B_{\infty,\infty}^{\rho}(2\Omega)}a^{j\sigma}a^{k\tau} \sum_{(\beta,\gamma)\in J_{j,i,k, l}}a^{-\beta\rho}.
		\end{align*}
		Hence
		$$\|I_{j,i,k,l}\| \leq  K\|\varphi\|_{B_{\infty,\infty}^{\rho}(2\Omega)} a^{j\sigma}a^{k\tau}a^{-\beta\rho},$$ 
		where $\beta = \beta(j,i,k,l)$ is as in \eqref{eq:Tsigmataurho}.
		Let $B = B(L^2(\Omega))$ denote the unit ball of $L^2(\Omega)$. Then 
		\begin{align*} \|\Ha_{\varphi}^{\sigma,\tau}\|&= \sup_{f,g\in B}|\sum_{j,i,k,l}\langle I_{j,i,k,l}f,g \rangle| 
			= \sup_{f,g\in B}|\sum_{j,i,k,l}\langle I_{j,i,k,l}f\chi_{E_{j}^{i}},g\chi_{E_{k}^{l}}\rangle| \nonumber \\
			&\leq \sup_{f,g\in B}\sum_{j,i,k,l} \|I_{j,i,k,l}\|\|f\chi_{E_{j}^{i}}\|_{L^{2}}\|g\chi_{E_{k}^{l}}\|_{L^{2}}\nonumber \\
			&\leq K\|\varphi\|_{B_{\infty,\infty}^{\rho}(2\Omega)}\sup_{\|a\|_{\ell^{2}},\|b\|_{\ell^{2}}\leq1}\sum_{j,i,k,l} a^{j\sigma}a^{k\tau}a^{-\beta\rho}|a_{j,i}||b_{k,l}|\nonumber \\
			&= K\|\varphi\|_{B_{\infty,\infty}^{\rho}(2\Omega)}\sup_{\|a\|_{\ell^{2}},\|b\|_{\ell^{2}}\leq1} \langle T_{\sigma,\tau,\rho}|\overline{a}|,|\overline{b}|\rangle =K \|\varphi\|_{B_{\infty,\infty}^{\rho}(2\Omega)}\|T_{\sigma,\tau,\rho}\|\nonumber,
		\end{align*} 
		where $|\overline{a}|$ denotes the sequence $(|a_{j,i}|)_{j,i}$.
	\end{proof}
	We now give a continuous version of Lemma \ref{Sinfty}. Importantly, the continuous formulation is independent of the decomposition $\{A^i_j\}$.
	\begin{theo}\label{Sinfty2}
		Let $\Omega$ be an admissible subset of $\mathbb{R}^n$ with decomposition $\{A_j^i\}$, and let $\sigma, \tau, \rho \in \mathbb{R}$. If 
		\[K_{\sigma, \tau,\rho}(x,y)=\frac{\omega_{\Omega}^{\sigma}(x)\omega_{\Omega}^{\tau}(y)}{\omega_{\Omega}^{\rho}(\frac{x+y}{2})}\]
		is the kernel of a bounded integral operator on $L^2(\Omega)$, then the operator $T_{\sigma+\frac{1}{2},\tau+\frac{1}{2},\rho}$ introduced in \eqref{eq:Tsigmataurho} is bounded. 
	\end{theo}
	\begin{proof}
		Let $\varphi_j^i$ be the partition of Proposition~\ref{exist}, and let $\psi_j^i=\varphi_j^i/\|\varphi_j^i\|_{L^2}$. Its construction gives $\|\widehat\psi_j^i\|_{L^1}\approx a^{j/2}$. Then
		$$\langle K_{\sigma, \tau,\rho}(\widehat{\psi}_j^i),\widehat{\psi}_k^l\rangle=\int_{\mathbb{R}^{n}}\int_{\mathbb{R}^{n}}\frac{\omega_{\Omega}^{\sigma}(x)\omega_{\Omega}^{\tau}(y)}{\omega_{\Omega}^{\rho}(\frac{x+y}{2})}\widehat{\psi}_j^i(x)\widehat{\psi}_k^l(y)dxdy\approx a^{j(\sigma+\frac{1}{2})}a^{k(\tau+\frac{1}{2})}a^{-\beta\rho},$$
		where $\beta = \beta(j,i, k, l)$ is as in \eqref{eq:Tsigmataurho}.
		Therefore, for $x, y \in \ell^2(\mathbb{Z}^2)$,
		\begin{eqnarray}
			|\langle T_{\sigma+\frac{1}{2},\tau+\frac{1}{2},\rho}x,y\rangle|&\lesssim& \sum_{j,i}\sum_{k,l}\langle K(\widehat{\psi}_j^i),\widehat{\psi}_k^l\rangle|x_{j,i}||y_{k,l}| \nonumber \\
			&=&\langle K(\sum_{j,i}\widehat{\psi}_j^i|x_{j,i}|),\sum_{k,l}\widehat{\psi}_k^l|y_{k,l}|\rangle \nonumber \\
			&\leq& \|K\|\|\sum_{j,i}\widehat{\psi}_j^i|x_{j,i}|\|_{L^2}\|\sum_{k,l}\widehat{\psi}_k^l|y_{k,l}|\|_{L^2} \nonumber \\
			&\lesssim& \|K\|\|x\|_{\ell^2}\|y\|_{\ell^2}, \nonumber
		\end{eqnarray}
		where the last inequality holds since $\psi_{j,i}$ are almost disjoint in the sense of property \ref{5} and normalized.
	\end{proof}
	\section{Necessity of the Besov space condition}\label{sec:nec}
	In this section we prove that the Besov space condition is necessary, for the full range $1 \leq p \leq \infty$ and all $\sigma, \tau \in \mathbb{R}$. We begin with a lemma which follows along the lines of the construction of $B_{p,q}^{s}(\Omega)$, cf. Lemma~\ref{Besov properties}.
	\begin{lm}\label{besov on 2}
		Let $\Omega$ be an admissible subset of $\mathbb{R}^{n}$ with decomposition $\{A_{j}^{i}\}$, and equip $2\Omega$ with the decomposition $\{2A_{j}^{i}\}$. There exists a family $(\varphi_{j}^{i})$ of smooth functions with $\supp\widehat{\varphi}_{j}^{i}\subset A_{j}^{i}$ and $\sup_{j, i} \|\varphi_j^i\|_{L^{1}(\mathbb{R}^n)} < \infty$, such that the family $(\psi_j^i)$, $\psi_{j}^{i}=a^{-j}(\varphi_{j}^{i})^{2}$, generates the Besov space $B_{p,q}^{s}(2\Omega)$, in the sense that
		\[
		C_1\|f\|_{B_{p,q}^{s}(2\Omega)} \leq \left(\sum_{j,i}a^{qjs}\|f\ast\psi^{i}_{j}\|^{q}_{L^{p}}\right)^{\frac{1}{q}} \leq C_2 \|f\|_{B_{p,q}^{s}(2\Omega)}, \qquad f \in X_{p,q}^{s}(2\Omega).
		\]
	\end{lm}
	\begin{proof}
		Let $\varphi$ and $\varphi_{j}^{i}$, be defined as in the proof of Proposition \ref{exist}, and set $\psi=\sum_{j,i}\widehat{\psi}_{j}^{i}=\sum_{j,i}a^{-j}\widehat{\varphi}_{j}^{i}\ast \widehat{\varphi}_{j}^{i}(x)$.  The properties of $T_{j,i}$ ensure that $|\det T_{j,i}|^{-1}=m(A_{j}^{i})\approx a^j$, and thus
		$$\widehat{\varphi}_{j}^{i}\ast \widehat{\varphi}_{j}^{i}(x)=|\det T_{j,i}|^{-1}\widehat{\varphi}\ast\widehat{\varphi}\left(2T_{j,i}(x/2)\right)\approx a^{j} \widehat{\varphi}\ast\widehat{\varphi}\left(2T_{j,i}(x/2)\right).$$ 
		Combined with property \ref{5}, it follows that $\psi$ is bounded from above. We also see that $\psi\gtrsim \epsilon^{n}>0$ on $2\Omega$; since $\widehat{\varphi}\geq \chi_{(-\frac{1-\epsilon}{2},\frac{1-\epsilon}{2})^{n}}$, a simple computation gives that 
		\[
		\widehat{\varphi}\ast\widehat{\varphi}(x)\geq (1-\epsilon-|x_{1}|)...(1-\epsilon-|x_{n}|)\geq \epsilon^{n}, \qquad x\in 2(-\frac{1}{2}+\epsilon,\frac{1}{2}-\epsilon)^{n}.
		\]
		We now consider $(f_j^i)$ given by $\widehat{f}_{j}^{i}=\dfrac{\widehat{\psi}_{j}^{i}}{\psi}$. The same argument as in Proposition~\ref{exist} shows that $(f_j^i) \in \mathscr{P}(\{2A^i_j\})$; it suffices to let $2T_{j,i}(\frac{\cdot}{2})$ take the role of $T_{j,i}$. By Lemma~\ref{Besov properties}, it therefore suffices to prove that $(\psi_{j}^{i})$ and $(f_{j}^{i})$ generate the same Besov space.
		
		In one direction we use that $\sup_{j.i} \|\psi^i_j\|_{L^1} < \infty$, and thus
		$$\|f\ast\psi_{j}^{i}\|_{L^{p}}=\|f\ast\psi_{j}^{i}\ast(\sum_{(k,l)\in J_{j,i}}f_{k}^{l})\|_{L^{p}} \leq K \sum_{(k,l)\in J_{j,i}}\|f\ast f_{k}^{l}\|_{L^{p}}.$$ 
		Arguing as in the proof of the first part of Lemma \ref{Besov properties}, we conclude that  
		$$\left(\sum_{j,i}a^{qjs}\|f\ast\psi^{i}_{j}\|^{q}_{L^{p}}\right)^{\frac{1}{q}} \leq C_2 \|f\|_{B_{p,q}^{s}(2\Omega)}.$$ 
		For the other inclusion, define $(h^i_j)$ by $\widehat{h}^i_j = \dfrac{\widehat{f}_{j}^{i}}{\psi}$. Just as for $f^i_j$, the argument of Proposition~\ref{exist} again shows that $\sup_{j,i} \|h_j^i\|_{L^1} < \infty$. Since $\widehat{\psi}_{j}^{i}=\sum_{(k,l)\in J_{j,i}}\widehat{\psi}_{j}^{i}\widehat{f}_{k}^{l}$, we conclude that
		$$\|f\ast f_{j}^{i}\|_{L^{p}}=\|f\ast\mathscr{F}^{-1}( \widehat{\psi}_{j}^{i}/\psi)\|_{L^{p}}\leq\sum_{(k,l)\in J_{j,i}}\|f\ast\psi_{j}^{i}\ast h_k^l\|_{L^{p}}\leq K \|f\ast\psi_{j}^{i}\|_{L^{p}},$$ 
		which completes the proof.
	\end{proof}
	
	The following lemma is essentially \cite[Lemma 3]{MR837521}, adjusted to our setting.
	\begin{lm}\label{Timotin lemma}
		Let $\alpha,\alpha'\in L^{\infty}(\mathbb{R}^{n})$, with $\supp \alpha \subset K_1$ and $\supp \alpha' \subset K_2$, where $K_1$ and $K_2$ are compact, let $A \in L^2(\mathbb{R}^{2n})$ be a kernel, and let 
		$$\widehat{b}(u)=\int_{\mathbb{R}^{n}}A(u-x,x)\alpha(u-x)\alpha'(x) \, dx.$$ 
		Then 
		$$\|b\|_{L^{1}}\leq \|\alpha\|_{L^{\infty}}\|\alpha'\|_{L^{\infty}}\| \chi_{K_2} A\chi_{K_1}\|_{S^{1}}.$$
	\end{lm}
	\begin{proof}
		We may assume that $A$ is supported in $K_1 \times K_2$, and, since $b$ depends linearly on $A$, that $A$ is of rank $1$, $A(x,y)=f(x)g(y)$. Then $\widehat{b} = (\alpha f)\ast(\alpha' g)$, and thus 
		\begin{equation*}
			\|b\|_{L^{1}} \leq\|\alpha f\|_{L^{2}}\|\alpha' g\|_{L^{2}} \leq  \|\alpha\|_{L^{\infty}}\|\alpha'\|_{L^{\infty}}\| f\|_{L^{2}}\| g\|_{L^{2}} 
			= \|\alpha\|_{L^{\infty}}\|\alpha'\|_{L^{\infty}}\|A\|_{S^{1}}. \qedhere
		\end{equation*}
	\end{proof} 
	We now prove Theorem~\ref{thm:mainnec}, stated more precisely than in the introduction.
	\begin{theo}\label{converse}
		Let $\Omega$ be an admissible subset of $\mathbb{R}^{n}$ with decomposition $\{A_{j}^{i}\}$, equip $2\Omega$ with the decomposition $\{2A_{j}^{i}\}$, and let $1 \leq p \leq \infty$. Suppose that $\Ha^{\sigma,\tau}_{\varphi}\in S^{p}(\PW(\Omega))$ for some $\sigma,\tau\in\mathbb{R}$. Then $\varphi\in B_{p,p}^{\sigma+\tau+\frac{1}{p}}(2\Omega)$, and there exists a constant $C > 0$ such that \[\|\varphi\|_{B_{p,p}^{\sigma+\tau+\frac{1}{p}}(2\Omega)}\leq C \|\Ha_{\varphi}^{\sigma,\tau}\|_{S^{p}}.\]
	\end{theo}
	
	\begin{proof}
		Letting $P^{i}_{j}$ be the projection of $L^{2}(\mathbb{R}^n)$ onto $\PW(A^{i}_{j})$, for $p<\infty$ property~\ref{5} gives $$\|\Ha_{\varphi}^{\sigma,\tau}\|_{S^{p}}^{p}\geq C\sum_{j,i}\|P^{i}_{j}\Ha_{\varphi}^{\sigma,\tau}P^{i}_{j}\|_{S^{p}}^{p},$$ while for $p=\infty$ we use $\|\Ha_{\varphi}^{\sigma,\tau}\|\geq\sup_{j,i}\|P_j^i\Ha_{\varphi}^{\sigma,\tau}P_j^i\|$.
		With $\varphi_{j}^{i}$ and $\psi_{j}^{i}$ as in Lemma \ref{besov on 2},
		it suffices to prove that 
		\begin{equation} \label{eq:necineq}
			\|P_{j}^{i}\Ha_{\varphi}^{\sigma,\tau}P_{j}^{i}\|_{S^{p}}\geq Ca^{j(\sigma+\tau+\frac{1}{p})}\|\varphi\ast\psi_{j}^{i}\|_{L^{p}}.
		\end{equation}
		We prove this inequality by using interpolation. 
		
		For $p=1$,  \eqref{eq:necineq}  falls out of Lemma~\ref{Timotin lemma}, applied with \[A(x,y)=\widehat{\varphi}(x+y)\omega_{\Omega}^{\sigma}(x)\omega_{\Omega}^{\tau}(y),
		\]
		$\alpha(x)=a^{j\sigma}\widehat{\varphi}_{j}^{i}(x)\omega_{\Omega}^{-\sigma}(x),$ and $\alpha'(x)=a^{j\tau}\widehat{\varphi}_{j}^{i}(x)\omega_{\Omega}^{-\tau}(x)$.
		
		To treat the case $p=\infty$, we define 
		$$\widehat{f}_{j,i}^{z,\sigma}(x)=e^{2\pi ixz}\widehat{\varphi}_{j}^{i}(x)\omega_{\Omega}^{\sigma}(x), \qquad z \in \mathbb{R}^n.$$
		Then we get that
		\begin{eqnarray}
			\|P^{i}_{j}\Ha_{\varphi}^{\sigma,\tau}P^{i}_{j}\|&\geq& \dfrac{|\langle P^{i}_{j}\Ha_{\varphi}^{\sigma,\tau}P^{i}_{j}(f_{j,i}^{z,-\sigma}),f_{j,i}^{-z,-\tau}\rangle|}{\|f_{j,i}^{z,-\sigma}\|_{L^{2}}\|f_{j,i}^{-z,-\tau}\|_{L^{2}}} \nonumber \\
			&\geq& C a^{j(\sigma+\tau-1)}|\langle P^{i}_{j}\Ha_{\varphi}^{\sigma,\tau}P^{i}_{j}(f_{j,i}^{z,-\sigma}),f_{j,i}^{-z,-\tau}\rangle|\nonumber \\
			&=& C a^{j(\sigma+\tau-1)}\left|\iint\widehat{\varphi}(x+y)\widehat{\varphi}^{i}_{j}(x)\widehat{\varphi}^{i}_{j}(y)e^{2\pi i(x+y)z}dx \, dy\right|.\nonumber 
		\end{eqnarray}
		Here we make the change of variables $x+y\to y$ and $x\to x$ to obtain that
		\begin{eqnarray}
			\|P^{i}_{j}\Ha_{\varphi}^{\sigma,\tau}P^{i}_{j}\| &\geq &C a^{j(\sigma+\tau-1)}\left|\iint\widehat{\varphi}(y)\widehat{\varphi}^{i}_{j}(x)\widehat{\varphi}^{i}_{j}(y-x)e^{2\pi iyz}dxdy\right|\nonumber \\
			&=& Ca^{j(\sigma+\tau)}|\varphi\ast\psi_{j}^{i}(z)|, \qquad z \in \mathbb{R}^n. \nonumber
		\end{eqnarray}
		This is \eqref{eq:necineq} for $p=\infty$.
	\end{proof}
	Corollary \ref{1 to 2 suf} and Theorem \ref{converse} immediately yield the following characterization in the regime $1 \leq p \leq 2$, which is Corollary~\ref{cor:mainchar}. 
	\begin{cor}\label{cor 1 to 2}
		Let $\Omega$ be an admissible subset of $\mathbb{R}^{n}$ with decomposition $\{A_{j}^{i}\}$, and equip $2\Omega$ with the decomposition $\{2A_{j}^{i}\}$. Then, for $1\leq p\leq 2$ and $\sigma,\tau>-\frac{1}{n+1}$ with $\sigma+\tau>-\frac{2}{p(n+1)}$ it holds true that
		$$\varphi\in B_{p,p}^{\sigma+\tau+\frac{1}{p}}(2\Omega) \textrm{ if and only if } \Ha_{\varphi}^{\sigma,\tau}\in S^{p}(\PW(\Omega)).$$ If $\Omega$ is a polyhedron the same holds true for $\sigma,\tau>-\frac12$ and $\sigma+\tau>-\frac1p$.
	\end{cor}
	
	It was recently shown \cite{bampouras2023failure} that the Nehari theorem fails for all convex bounded sets $\Omega$ which are not polytopes. That is, there are bounded Hankel operators $\Ha_{\varphi}\in S^{\infty}(\PW(\Omega))$ for which there exists no $\psi \in L^\infty(\mathbb{R}^n)$ such that $\Ha_{\varphi} = \Ha_\psi$. On the other hand, by a duality argument, it is easy to see that $B_{1,1}^1(2\Omega) \subset L^\infty(\mathbb{R}^n)$. Therefore every symbol $\varphi$, with Fourier support in $2\Omega$, which induces a trace class Hankel operator $\Ha_{\varphi}\in S^{1}$ must in fact be bounded, and thus the conclusion of the Nehari theorem holds trivially in this case. It is an interesting question, raised in \cite{BP23}, whether the conclusion of the Nehari theorem holds when $\Ha_{\varphi}\in S^{p}$, $1 < p < \infty$.
	
	\section{Smooth domains in $\mathbb{R}^n$ with positive curvature} \label{sec:smooth}
	In this section we assume throughout that $\Omega \subset \mathbb{R}^n$ is a bounded, convex, $C^{2}$-smooth domain with boundary that has strictly positive curvature. Such domains are strongly convex. By translating, we may assume that $0 \in \Omega$.
	
	To demonstrate that $\Omega$ is admissible, we will essentially show that any reasonable choice of parallelepipeds $\{A^i_j\}$ respecting properties \ref{1}-\ref{3} will automatically satisfy properties \ref{4} and \ref{5} as well. We will then apply Theorem~\ref{S2} and Theorem~\ref{Sinfty2} to obtain Schatten class estimates and prove Theorem~\ref{thm:mainsmooth}.
	
	We start by giving an approximation of $\omega_\Omega$, based on the simple computation that for the $n$-dimensional unit ball $B_n$,
	\begin{equation} \label{eq:omegadisc}
		\omega_{B_n}(x)\approx \left(1-|x|\right)^{\frac{n+1}{2}}, \qquad x \in B_n.
	\end{equation}
	\begin{lm}\label{omegaconvex}
		It holds that
		$$\omega_{\Omega}(x)\approx (\dist(x,\partial\Omega))^{\frac{n+1}{2}}, \qquad x\in \Omega.$$
	\end{lm}
	\begin{proof}
		By hypothesis, the principal curvatures of $\partial \Omega$ reach a minimum $k>0$ and a maximum $K>0$. Given $z\in\partial \Omega$, the balls with radii $\frac{1}{K}$ and $\frac{1}{k}$, tangent to $z$, are inscribed and circumscribed, respectively.
		
		Let $x\in\Omega$ be such that $\dist(x,\partial \Omega)<\frac{1}{K}$, and let $z\in\partial\Omega$ be such that $|x-z|=\dist(x,\partial \Omega)$.  Denote the inscribed and circumscribed balls at $z$ of radii $\frac{1}{K}$ and $\frac{1}{k}$ by $B_{1}$ and $B_{2}$, respectively, and denote their centers by $y_{1}$ and $y_{2}$.  Since $z,y_{1},y_{2},x$ all lie on the line orthogonal to $\partial \Omega$ at $z$, we have that $$\dist(x,\partial\Omega)=\frac{1}{K}-|x-y_{1}|=\frac{1}{k}-|x-y_{2}|,$$
		and therefore, by \eqref{eq:omegadisc},
		$$\omega_{B_1}(x) \approx \big(\frac{1}{K}-|y_{1}-x|\big)^{\frac{n+1}{2}} = \dist(x,\partial\Omega)^{\frac{n+1}{2}} = \big(\frac{1}{k}-|y_{2}-x|\big)^{\frac{n+1}{2}} \approx \omega_{B_2}(x).$$
		Since $B_{1}\subset \Omega\subset B_{2}$, we have that
		$\omega_{B_{1}}\leq \omega_{\Omega}\leq\omega_{B_{2}}$, and thus $\omega_{\Omega}(x) \approx \dist(x,\partial\Omega)^{\frac{n+1}{2}}$.
	\end{proof}
	\begin{lm}\label{conection to disc}
		Let $\lambda \colon \mathbb{R}^n \to [0,\infty)$ be the gauge function such that $$x\in\lambda(x)\partial\Omega, \qquad x \in \mathbb{R}^n.$$ 
		Then
		$$\dist(x,\partial\Omega)\approx 1-\lambda(x),\qquad x\in\Omega,$$
		and therefore $\omega_{\Omega}(x)\approx (1-\lambda(x))^{\frac{n+1}{2}}.$
	\end{lm}
	\begin{proof}
		The upper bound is trivial since $$\dist(x,\partial\Omega)\leq |x-\frac{x}{\lambda(x)}|=\frac{|x|}{\lambda(x)}(1-\lambda(x)) \leq \max_{y\in\partial\Omega}|y|(1-\lambda(x)).$$ For the converse, let us assume for contradiction that for every $j\geq 1$, we can find $x_j\in\Omega$, $y_j\in\partial\Omega$ such that $|x_j-y_j|\leq \frac{1}{j}(1-\lambda(x_j))$, and thus $\frac{|\lambda(y_j)-\lambda(x_j)|}{|y_j-x_j|}\geq j$ (since $\lambda(y_j)=1)$. Since $\Omega$ has a $C^2$ boundary, $\lambda \in C^2(\mathbb{R}^n \setminus\{0\})$, and thus by the mean value theorem we get contradiction.
	\end{proof}
	We also need to clarify the behavior of $\omega_{\Omega}$ with respect to addition.
	\begin{prop}\label{thesum}
		For $x,y\in \Omega$, let $\theta(x,y)$ be the positive angle between the vectors $x$ and $y$. Then it holds that
		$$\omega_{\Omega}\big(\dfrac{x+y}{2}\big)\approx \max(\omega_{\Omega}(x),\omega_{\Omega}(y),\theta(x,y)^{n+1}).$$
	\end{prop}
	\begin{proof}
		We first establish the upper bound. Since $\omega_{\Omega}$ is continuous and non-vanishing in the interior of $\Omega$, it suffices by compactness to consider $\lambda(x),\lambda(y)>\frac{1}{2}$ and $\theta(x,y) < \epsilon$, where $\epsilon$ is chosen sufficiently small to ensure that $|x+y|$ is bounded away from zero. By Lemma \ref{conection to disc}, we can estimate
		\begin{eqnarray}
			\dist\big(\dfrac{x+y}{2},\partial\Omega\big)&\approx& 1-\lambda(\frac{x+y}{2}) \nonumber \\ 
			&\leq& 1-\frac{\lambda(x)+\lambda(y)}{2}+\big|\frac{\lambda(x)+\lambda(y)}{2}-\lambda(\frac{x+y}{2})\big|, \nonumber 
		\end{eqnarray}
		where, since $\lambda \in C^2(\overline{\Omega}\smallsetminus D)$ for every open disc $D$ containing $0$, 
		$$\big|\frac{\lambda(x)+\lambda(y)}{2}-\lambda(\frac{x+y}{2})\big|\lesssim |x-y|^{2}.$$ 
		Furthermore, recalling that $\frac{x}{\lambda(x)}\in\partial\Omega$, we have that $|\frac{x}{\lambda(x)}-\frac{y}{\lambda(y)}|\approx \theta(x,y)$, and thus by the triangular inequality,
		$$|x-y|\lesssim |x-\frac{x}{\lambda(x)}|+|\frac{x}{\lambda(x)}-\frac{y}{\lambda(y)}|+|y-\frac{y}{\lambda(y)}|\approx 1-\lambda(x)+1-\lambda(y)+\theta(x,y).$$
		Therefore we get that
		$$\big|\frac{\lambda(x)+\lambda(y)}{2}-\lambda(\frac{x+y}{2})\big|\lesssim \max((1-\lambda(x))^2,(1-\lambda(y))^2,\theta^2(x,y)).$$
		The upper bound now follows from the estimate
		\begin{eqnarray}
			1-\frac{\lambda(x)+\lambda(y)}{2}&\approx& 1-\left(\frac{\lambda(x)+\lambda(y)}{2}\right)^{2} = 1-\frac{\lambda^2(x)}{4}-\frac{\lambda^{2}(y)}{4}-\frac{\lambda(x)\lambda(y)}{2} \nonumber  \\
			&=& \frac{1-\lambda^{2}(x)}{4}+\frac{1-\lambda^{2}(y)}{4}+\frac{1-\lambda(x)\lambda(y)}{2} \nonumber \\
			&\approx& 1-\lambda(x)+1-\lambda(y). \nonumber 
		\end{eqnarray}
		For the lower bound, the concavity of $\omega^{\frac{1}{n}}$ implies that
		\begin{equation} \label{eq:omegaconcuse}
			\begin{aligned}
				\omega^{\frac{1}{n}}_{\Omega}(\frac{x+y}{2})&\geq \frac{1}{2}\omega^{\frac{1}{n}}_{\Omega}(x)+\frac{1}{2}\omega^{\frac{1}{n}}_{\Omega}(y)\approx (1-\lambda(x))^{\frac{n+1}{2n}}+(1-\lambda(y))^{\frac{n+1}{2n}}  \\
				&\approx \max( 1-\lambda(x),1-\lambda(y))^{\frac{n+1}{2n}}.
			\end{aligned}
		\end{equation}
		To complete the proof it therefore suffices to show that 
		\begin{equation} \label{eq:sumdistcontr} \omega_{\Omega}(\frac{x+y}{2})\gtrsim \theta^{n+1}(x,y),
		\end{equation}
		or, equivalently, $$\dist(\frac{x+y}{2},\partial\Omega)\gtrsim \theta^{2}(x,y).$$
		Note that if we consider $x,y \in\Omega$ such that $\max(1-\lambda(x),1-\lambda(y))\geq K\theta^2(x,y)$, then, by \eqref{eq:omegaconcuse},
		$$\dist\big(\frac{x+y}{2},\partial\Omega\big)\gtrsim K\theta^{2}(x,y).$$
		Therefore, if the lower bound \eqref{eq:sumdistcontr} does not hold, we may find $x_{j}\to x$ and $y_{j}\to y$ such that
		\begin{equation} \label{eq:contr1}
			\dist\big(\frac{x_j+y_j}{2},\partial\Omega\big)\leq \frac{1}{j}\theta^{2}(x_j,y_j)
		\end{equation}
		and
		\begin{equation} \label{eq:contr2}
			\max(1-\lambda(x_j),1-\lambda(y_j)) \leq \frac{1}{j}\theta^{2}(x_j,y_j).
		\end{equation}
		Let $z_{j}\in\partial\Omega$ achieve the minimum distance, 
		\begin{equation} \label{eq:contr3}
			\dist\big(\frac{x_j+y_j}{2},\partial\Omega\big)=\big|\frac{x_j+y_j}{2}-z_j\big|
		\end{equation} 
		and assume, without any loss of generality, that $z_j$ converges to some $z$. Then
		$\frac{x+y}{2}=z$, which by strict convexity implies that $x=y=z$.
		
		By \eqref{eq:contr1} and \eqref{eq:contr3}, we find that
		$$\big|\frac{x_j+y_j}{2} - z_j\big| = o(\theta^2(x_j,y_j)),$$
		and thus
		\begin{equation}\label{eq:contr4}
			|\lambda(\frac{x_j+y_j}{2})- 1|\lesssim o(\theta^2(x_j,y_j)).
		\end{equation}
		
		On the diagonal, the Hessian of the function $f(x,y)=\dfrac{\lambda(x)+\lambda(y)}{2}-\lambda(\dfrac{x+y}{2})$ satisfies $$\langle Hf(x,x)(a,b),(a,b)\rangle=\frac{1}{4}\langle H\lambda(x)(a-b),a-b \rangle, \qquad a, b \in \mathbb{R}^n.$$
		
		Therefore, by Taylor expansion we have that
		\begin{equation}\label{eq:contr5} \dfrac{\lambda(x_j)+\lambda(y_j)}{2}-\lambda(\dfrac{x_j+y_j}{2}) = \frac{1}{8}\langle H\lambda(x)(x_j-y_j),x_j-y_j \rangle + o(|x_j - y_j|^2).
		\end{equation}
		Combining \eqref{eq:contr2},\eqref{eq:contr4} and \eqref{eq:contr5}, we conclude that
		\begin{eqnarray}
			|\langle H\lambda(x)(x_j-y_j),x_j-y_j \rangle|&\approx& |\dfrac{\lambda(x_j)+\lambda(y_j)}{2}-\lambda(\dfrac{x_j+y_j}{2})| \nonumber \\
			&\lesssim& 1-\lambda(x_j)+1-\lambda(y_j)+|1-\lambda(\dfrac{x_j+y_j}{2})|\nonumber \\
			&=& o(\theta^2(x_j,y_j)) =  o(|x_j-y_j|^2). \nonumber
		\end{eqnarray} 
		By compactness, we may assume that $\frac{x_j-y_j}{|x_j-y_j|}\to u\in S^{n-1}$. We have then just shown that
		$$\langle H\lambda(x)u,u \rangle=0.$$
		On the other hand, $u$ is orthogonal to $\nabla \lambda(x)$, since, by first order Taylor expansion and \eqref{eq:contr2}, we have that
		$$|\langle \nabla \lambda(x),x_j-y_j\rangle|\approx |\lambda(x_j)-\lambda(y_j)|=o(\theta^2(x_j,y_j))=o(|x_j-y_j|^2).$$ 
		These two facts are in contradiction; since $\Omega$ is strongly convex, the Hessian $H\lambda(x)$ must be strictly positive definite on the tangent space $T(x)$ at $x \in \partial \Omega$. 
	\end{proof}
	
	\begin{figure}[!h]
		\begin{tikzpicture}[>=Latex, line cap=round, line join=round, scale=1.0]
			
			\def\halfw{1.7}      
			\def\halfh{1.1}      
			\def\gap{1.2}        
			
			\coordinate (z) at (0,0);
			\draw[thick] plot[domain=-4.6:4.6,samples=200] (\x,{-0.08*(\x)^2}); 
			
			\draw[thick] (-5,0) -- (5,0);
			\node[fill=white,inner sep=1pt,anchor=west] at (3.6,0.22) {$T(z)$};
			
			\coordinate (x)  at (0,-\gap-\halfh);
			\coordinate (A)  at ($(x)+(-\halfw,-\halfh)$);
			\coordinate (B)  at ($(x)+(\halfw,-\halfh)$);
			\coordinate (C)  at ($(x)+(\halfw,\halfh)$);
			\coordinate (D)  at ($(x)+(-\halfw,\halfh)$);
			\draw[thick] (A) rectangle (C);
			
			\draw[thick,->] (x) -- (0,2.2);
			\node[fill=white,inner sep=1pt,anchor=east] at (-0.15,0.9) {$N(z)$};
			
			\fill (z) circle[radius=1.5pt];
			\node[fill=white,inner sep=1pt,anchor=west] at (0.08,0.12) {$z$};
			\fill (x) circle[radius=1.5pt];
			\node[fill=white,inner sep=1pt,anchor=north] at ($(x)-(0.08, 0.08)$) {$x$}; 
			
			\draw[decorate,decoration={brace,mirror,amplitude=6pt}] (A) -- (B)
			node[midway,below=7pt,fill=white,inner sep=1pt] {$2c\,a^{\tfrac{j}{n+1}}$};
			
			\draw[decorate,decoration={brace,mirror,amplitude=6pt}] (D) -- (A)
			node[midway,left=7pt,fill=white,inner sep=1pt] {$2c\,a^{\tfrac{2j}{n+1}}$};
			
			\draw[very thick,<->] (0.45,0) -- (0.45,{-\gap-\halfh})
			node[midway,right=7pt,fill=white,inner sep=1pt] {$\approx a^{\tfrac{2j}{n+1}}$};
			
			\node[anchor=west] at ($(C)+(0.2,-1)$) {$A_x$};
			
		\end{tikzpicture}
	\end{figure}
	Now we turn our attention to proving that every $C^2$ convex domain with non-vanishing boundary curvatures is actually admissible. We start with the following lemma.
	\begin{lm}\label{ex}
		There are $c>0$ and $k_0\in\mathbb{Z}_{+}$ such that for every $x\in \Omega$, we can find a rotation $R_{x}$ for which
		$$A_x := x+cR_{x} P_j \subset \bigcup_{|j-k|<k_0}\Delta_k(\Omega),$$
		where $j$ is such that $x\in \Delta_j(\Omega)$ and $P_j=(-a^{\frac{2j}{n+1}},a^{\frac{2j}{n+1}})\times (-a^{\frac{j}{n+1}},a^{\frac{j}{n+1}})^{n-1}$.
	\end{lm} 
	\begin{proof}
		Let us use the notation $e_1,...,e_n$ for the classical orthonormal basis of $\mathbb{R}^n$. It is enough to prove the lemma for sufficiently small $j$. Fix $x\in\Delta_j(\Omega)$ and let $z\in\partial\Omega$ be such that 
		\begin{equation} \label{eq:disttobdryugh1}
			\dist(x,\partial\Omega)=|x-z| \approx a^{\frac{2j}{n+1}}.
		\end{equation}
		Here we have implicitly used Lemma \ref{omegaconvex}.
		Let $R_x$ be a rotation such that $R_xe_1=\frac{z-x}{|z-x|}=N(z)$ is the outward-pointing normal vector at $z$. Let $B_1$ and $B_2$ be the inscribed and circumscribed balls at $z$ of radii $\frac{1}{K}$ and $\frac{1}{k}$, where $k$ and $K$ are the minimal and maximal principal curvatures on $\Omega$. By comparing with the distance to the balls $B_1$ and $B_2$ we find, for sufficiently small $t>0$ and $w\in \spans(R_x e_2,...,R_x e_n)$ with $t>\frac{1}{k}\|w\|^2$, that 
		\begin{equation} \label{eq:disttobdryugh2}
			\dist(z-tN(z)+w,\partial\Omega)\approx t.
		\end{equation}
		Therefore, if $j$ is sufficiently small, and we combine \eqref{eq:disttobdryugh1} and \eqref{eq:disttobdryugh2}, we see that there is a constant $c$ (depending only on $\Omega$) such that any $u\in x+cR_x P_j$ can be written $u=z-tN(z)+w$ with $t\approx a^{\frac{2j}{n+1}}$ and $\|w\|\lesssim  a^{\frac{j}{n+1}}$. By Lemma \ref{omegaconvex} again, we conclude that
		$$\omega_{\Omega}(u)\approx\dist(u,\partial\Omega)^{\frac{n+1}{2}}\approx t^{\frac{n+1}{2}}\approx a^{j},$$
		which is what we wanted to show.
	\end{proof}
	Next, we show that the parallelepipeds in Lemma~\ref{ex} automatically satisfy property \ref{4}. In the proof, we have to do some additional work related to the issue that in most dimensions it is impossible to make a continuous choice $x \mapsto R_x$.
	\begin{lm}\label{1-4}
		For $x \in \Omega$ and $0<\epsilon \leq c$, let $\tilde{A}_x=x+ \epsilon R_{x} P_j$, where $c,P_j$, and $R_{x}$ are as in Lemma \ref{ex}, and let $\tilde{T}_x$ be the affine transform that maps $\tilde{A}_x$ onto $(-1,1)^n$. Then there exists $C>0$, depending only on $\Omega$, such that if $y\in \Omega$ is a point for which $\tilde{A}_x\cap \tilde{A}_y\neq\emptyset$, then
		$$\tilde{T}_x \tilde{A}_y \subset C(-1,1)^n.$$
	\end{lm}
	\begin{proof}
		Let $j$ and $k$ be such that $x \in \Delta_j(\Omega)$ and $y \in \Delta_k(\Omega)$. Then, by Lemma \ref{ex}, $|j-k|\leq 2 k_0$. Let $$D_{j}=\diag (a^{\frac{2j}{n+1}},a^{\frac{j}{n+1}},...,a^{\frac{j}{n+1}}).$$ Let $u \in \tilde{A}_x\cap \tilde{A}_y$ be any point. By the definition of $\tilde{A}_y$, we have that 
		$$\tilde{A}_y-u \subset 2\epsilon R_{y}D_k (-1,1)^n.$$
		Applying $\tilde{T}_{x},$ noting that $\tilde{T}_x z=\epsilon^{-1}D_j^{-1}R_{x}^{-1}(z-x),$ we see that
		$$\tilde{T}_x (\tilde{A}_y)\subset \tilde{T}_xu + 2D_j^{-1}R_{x}^{-1}R_{y}D_k (-1,1)^n,$$
		where, since $u\in \tilde{A}_x$, $\tilde{T}_x u\in (-1,1)^n$.
		
		It therefore suffices to show that the operator $ D_j^{-1}R_{x}^{-1}R_{y}D_k$ is bounded independently of $x,y,j$, and $k$, whenever $\tilde{A}_x \cap \tilde{A}_y \neq \emptyset$ and (thus) $|j-k|\leq 2k_0$. Let $z_1,z_2\in\partial\Omega$ be points for which $|x-z_1|=\dist(x,\partial\Omega)$ and $|y-z_2|=\dist(y,\partial\Omega)$, so that, by definition, $R_{x}e_1=N(z_1)$ and $R_{y}e_1=N(z_2)$.
		Note that the rotations $R_x$ and $R_y$ are only fixed up to composition with rotations $U$ that fix $e_1$. Let $\tilde{R}_{x,y}$ be the rotation along the shorter segment of the great circle connecting $N(z_1)$ and $N(z_2)$, so that
		$$\|I - \tilde{R}_{x,y}\| \lesssim |N(z_1)-N(z_2)|\lesssim |z_1-z_2|\leq |z_1-x|+|x-y|+|y-z_2|\lesssim a^{\frac{j}{n+1}}.$$
		Furthermore, there are rotations $U_1$ and $U_2$ such that $U_1 e_1 = U_2 e_1 = e_1$ and $U_1 R_x^{-1} R_y U_2 = \tilde{R}_{x,y}$, and therefore
		\[\|D_j^{-1}R_{x}^{-1}R_{y}D_k\| \leq \|D_j^{-1} U_1^{-1} D_k\| \|D_k^{-1} \tilde{R}_{x,y} D_j \| \| D_j^{-1} U_2^{-1} D_k\|\]
		Note that $U_j^{-1} = I \oplus \tilde{U}_j^{-1}$ splits into an orthogonal sum of the identity on $\textrm{span} \{e_1\}$ and a rotation of the subspace $\textrm{span} \{e_2, \ldots, e_n\}$, on which the operators $D_j$ and $D_k$ act as multiples of the identity. Since $|j-k| \leq 2k_0$, we therefore find that
		\begin{align*}
			\|D_j^{-1}R_{x}^{-1}R_{y}D_k\| &\lesssim \|D_k^{-1} \tilde{R}_{x,y} D_j \| \leq \|D_k^{-1} D_j\| + \|D_k^{-1}\| \|I -\tilde{R}_{x,y}\| \|D_j\| \\
			&\lesssim 1+ a^{\frac{-2k}{n+1}} a^{\frac{j}{n+1}} a^{\frac{j}{n+1}}\approx 1. \qedhere
		\end{align*}
		
	\end{proof}
	We now describe how to pick points $(x_{j,i})$ to obtain a valid decomposition $\{A^i_j\}$.
	\begin{lm}\label{existxji}
		There exists a sequence $(x_{j,i})_{j,i} \subset \Omega$ and $0<\tilde{\epsilon}<\epsilon<c$ such that
		\begin{enumerate} 
			\item \label{one} $(x_{j,i}+\tilde{\epsilon} R_{x_{j,i}}P_j)\cap (x_{k,l}+\tilde{\epsilon} R_{x_{k,l}}P_k)= \emptyset$ whenever $(j,i) \neq (k,l)$, and
			\item \label{two} $x_{j,i}+\epsilon R_{x_{j,i}}P_j$ cover $\Omega$,
		\end{enumerate} 
		where $P_j,R_{x}$, and $c$ are defined as in Lemma \ref{ex}.
	\end{lm}
	
	\begin{proof}
		Let us fix $\tilde{\epsilon}<\frac{c}{1+C}$, where $C$ is as in Lemma \ref{1-4}. By Zorn's lemma we can pick a maximal sequence $(x_{j,i})_{j,i}$ such that (\ref{one}) holds. Let $y\in \Delta_j(\Omega)$. The maximality of $(x_{k,l})_{k,l}$ means that the parallelepiped $y+\tilde{\epsilon} R_{y}P_j$ must intersect a parallelepiped $x_{k,l}+\tilde{\epsilon} R_{x_{k,l}}P_k$. But then Lemma \ref{1-4} says that
		$$y+\tilde{\epsilon} R_{x_{j,i}}P_j \subset x_{k,l}+C \tilde{\epsilon} R_{x_{k,l}}P_k.$$
		Thus (\ref{two}) holds as well, for $\epsilon=C\tilde{\epsilon}$. 
	\end{proof}
	Fix a sequence $(x_{j,i})$ as in Lemma~\ref{existxji}, and let 
	\begin{equation} \label{eq:aijdef}
		A_j^i := A_{x_{j,i}}.
	\end{equation}
	Then Lemma~\ref{ex} shows that properties \ref{1} and \ref{2} hold, Lemma~\ref{existxji} (2) shows that \ref{3} holds, while Lemma~\ref{1-4} shows that \ref{4} holds. In order to prove that $\{A^i_j\}$ is a decomposition for $\Omega$, it thus only remains to verify property \ref{5}.
	\begin{lm}\label{admis}
		The sets $A_j^i$ defined in \eqref{eq:aijdef} satisfy property \ref{5}. 
	\end{lm}
	\begin{proof}
		We fix $j,i$ and $k,l$ throughout the proof, as we estimate the cardinality of
		$$J_{j,i,k,l}=\{(\beta',\gamma'):A_{\beta'}^{\gamma'}\cap(\frac{1}{2}(A_j^i+A_k^l))\neq \emptyset\}.$$
		Furthermore, fix a pair $\beta,\gamma$ such that $(A_j^i+A_k^l)\cap 2A_\beta^\gamma\neq\emptyset$. Then, by Proposition \ref{thesum}, for any $x\in A_j^i, y\in A_k^l$, 
		$$\omega_{\Omega}(\frac{x+y}{2})\approx \max(a^j,a^k,\theta(x,y)^{n+1}) \approx \max(a^j,a^k,\theta(x_{j,i},x_{k,l})^{n+1}),$$
		where the second equivalence follows from the fact that
		$$\theta(x,y)\lesssim \theta(x,x_{j,i})+\theta(x_{j,i},x_{k,l})+\theta(x_{k,l},y)\lesssim \max(a^{\frac{j}{n+1}},a^{\frac{k}{n+1}},\theta(x_{j,i},x_{k,l})).$$ 
		Therefore 
		$$a^\beta\approx \max(a^j,a^k,\theta(x_{j,i},x_{k,l})^{n+1}).$$ 
		In particular, there is a constant $M_0$ such that if $(A_j^i+A_k^l)\cap 2A_{\beta'}^{\gamma'}\neq\emptyset$, then $|\beta - \beta'| \leq M_0$, so that 
		$$\frac{1}{2}(A_j^i+A_k^l) \subset \omega^{-1}_{\Omega}(a^{\beta - M}, a^{\beta + M}),$$
		where $M = M_0 + 1$. Furthermore, since $\frac{1}{2}(A_j^i+A_k^l)$ is convex, by the hyperplane separation theorem there is an $x = x(i,j,k,l) \in \omega_{\Omega}^{-1}(a^{\beta+M})$ such that $\frac{1}{2}(A_j^i+A_k^l)\subset \Lambda_{a^{2M}}(x)$, where $\Lambda_{a^{2M}}(x)$ is the convex cap defined in \eqref{eq:convcap}.
		
		For any $(\beta', \gamma') \in J_{j,i,k,l}$, there is thus by definition a point $z\in \Lambda_{a^{2M}}(x)\cap A_{\beta'}^{\gamma'}$. By convexity, $\frac{x+z}{2}\in \Lambda_{a^{2M}}(x)$, and therefore
		$$a^{\beta}\approx \omega_{\Omega}(\frac{x+z}{2}) \approx \max(\omega_{\Omega}(x),\omega_{\Omega}(z),\theta(x,z)^{n+1})\approx \max(a^{\beta},\theta(x,z)^{n+1}),$$ 
		and hence $\theta(x,z)\lesssim a^{\frac{\beta}{n+1}}$. We conclude that every point $y\in A_{\beta'}^{\gamma'}$ satisfies that 
		$$\theta(x,y)\lesssim \theta(x,z)+\theta(z,y)\lesssim a^{\frac{\beta}{n+1}}.$$ 
		
		To summarize, we have now shown that there are $M,C>0$ such that 
		$$\bigcup_{(\beta', \gamma') \in J_{j,i,k,l}} A_{\beta'}^{\gamma'}\subset \{y\in \Omega:a^{\beta-M}\leq \omega_{\Omega}(y)\leq a^{\beta+M},  \; \theta(x,y)\leq Ca^{\frac{\beta}{n+1}}\},$$
		where $x = x(i,j,k,l)$ is the point described earlier.
		Using the radial $C^1$-diffeomorphism $F(y)=\lambda(y)y/|y|$ from $\Omega\setminus\{0\}$ to $B_n\setminus\{0\}$, whose Jacobian is uniformly comparable to $1$ on this boundary layer, we can estimate the volume of this latter set,
		\begin{multline*} m(\{y\in \Omega:a^{\beta-M'}\leq \omega_{\Omega}(y)\leq a^{\beta+M'},\theta(x,y)\leq Ca^{\frac{\beta}{n+1}}\}) \\
			\approx  m(\{y\in B_n: C_1 a^{\frac{2\beta}{n+1}} \leq 1-|y| \leq C_2 a^{\frac{2\beta}{n+1}}, \; \theta(F(x),y)\leq C_3a^{\frac{\beta}{n+1}}\})\approx a^{\beta}. 
		\end{multline*}
		We finally apply Lemma~\ref{existxji} (1), to conclude that 
		\begin{align*}
			\# J_{j,i,k,l} &\lesssim a^{-\beta}m\left (\bigcup_{(\beta',\gamma')\in J_{j,i,k,l}} x_{\beta',\gamma'}+\tilde\epsilon R_{x_{\beta',\gamma'}}P_{\beta'}\right)\lesssim 1,
		\end{align*}
		because these smaller boxes are disjoint and their union is contained in a fixed enlargement of the preceding boundary layer.
		The proof is complete.
	\end{proof}
	
	To finish the proof of Theorem~\ref{thm:mainsmooth}, we need to verify the endpoint criterion given by Theorem~\ref{Sinfty2}.
	\begin{lm}\label{Kbound}
		For every $\sigma,\tau>-\frac{1}{n+1}$, the integral operator with kernel $$K(x,y)=\dfrac{\omega_{\Omega}^{\sigma}(x)\omega_{\Omega}^{\tau}(y)}{\omega_{\Omega}^{\sigma+\tau+1}(\frac{x+y}{2})},$$ is bounded on $L^2(\Omega)$.
	\end{lm}
	\begin{proof}
		The hypotheses on $\sigma$ and $\tau$ allow us to choose $\gamma$ such that
		$$-\frac{2}{n+1}-\sigma<\gamma<\sigma, \qquad -\frac{2}{n+1}-
		\tau<\gamma<\tau.$$
		We will apply Schur's test with weights $\omega_{\Omega}^{\gamma}(x)$ and $\omega_{\Omega}^{\gamma}(y)$.  By Lemma~\ref{omegaconvex} and Proposition \ref{thesum} we may assume that $\Omega = B_n$ is the unit ball, and we need to prove precisely that the integral 
		\begin{multline*}
			\int_{B_n} \dfrac{(1-|x|)^{\frac{n+1}{2}\sigma}(1-|y|)^{\frac{n+1}{2}\tau}}{\max((1-|x|)^{\frac{n+1}{2}},(1-|y|)^{\frac{n+1}{2}},\theta(x,y)^{n+1})^{\sigma+\tau+1}}\dfrac{(1-|x|)^{\frac{n+1}{2}\gamma}}{(1-|y|)^{\frac{n+1}{2}\gamma}}dx = \\
			\underbrace{(1-|y|)^{\frac{n+1}{2}(\tau-\gamma)}\int_{\{1-|x|>\max(1-|y|,\theta^2(x,y))\}}(1-|x|)^{\frac{n+1}{2}(\gamma-\tau-1)}dx}_{I_1} \\
			\qquad + \underbrace{(1-|y|)^{\frac{n+1}{2}(-\gamma-\sigma-1)}\int_{\{1-|y|>\max(1-|x|,\theta^2(x,y))\}}(1-|x|)^{\frac{n+1}{2}(\sigma+\gamma)}dx}_{I_2} \\
			+ \underbrace{(1-|y|)^{\frac{n+1}{2}(\tau-\gamma)}\int_{\{\theta^2(x,y)>\max(1-|x|,1-|y|)\}}\dfrac{(1-|x|)^{\frac{n+1}{2}(\sigma+\gamma)}}{\theta(x,y)^{(n+1)(\sigma+\tau+1)}}dx}_{I_3}
		\end{multline*}
		is bounded independently of $y$. 
		
		First we treat $I_1$. Note that, uniformly in $r$,
		$$m_{n-1}(|x|=r:\theta(x,y) < \rho) = O(\rho^{n-1}),$$
		where $m_{n-1}$ denotes $(n-1)$-dimensional measure.
		Thus, by Tonelli's theorem, 
		\begin{eqnarray}
			I_1&\lesssim&(1-|y|)^{\frac{n+1}{2}(\tau-\gamma)}\int_{\{r<|y|\}}(1-r)^{\frac{n+1}{2}(\gamma-\tau-1)}m_{n-1}(\{|x|=r:\theta(x,y)<\sqrt{1-r}\})dr \nonumber \\
			&\lesssim&(1-|y|)^{\frac{n+1}{2}(\tau-\gamma)}\int_{\{r<|y|\}}(1-r)^{\frac{n+1}{2}(\gamma-\tau-1)}(1-r)^{\frac{n-1}{2}}dr \nonumber \\
			&\approx&(1-|y|)^{\frac{n+1}{2}(\tau-\gamma)}(1-|y|)^{\frac{n+1}{2}(\gamma-\tau)}=1. \nonumber 
		\end{eqnarray}
		where we used that $\frac{n+1}{2}(\gamma-\tau-1)+\frac{n-1}{2} = \frac{n+1}{2}(\gamma-\tau) -1 <-1$. 
		
		Arguing similarly for $I_2$, 
		\begin{eqnarray}
			I_2&\lesssim&(1-|y|)^{\frac{n+1}{2}(-\gamma-\sigma-1)}\int_{r>|y|}(1-r)^{\frac{n+1}{2}(\sigma+\gamma)}m_{n-1}(\{|x|=r:1-|y|>\theta^2(x,y)\})dr \nonumber \\
			&\approx&(1-|y|)^{\frac{n+1}{2}(-\gamma-\sigma-1)}\int_{r>|y|}(1-r)^{\frac{n+1}{2}(\sigma+\gamma)}(1-|y|)^{\frac{n-1}{2}}dr \nonumber  \\
			&\approx&(1-|y|)^{\frac{n+1}{2}(-\gamma-\sigma)-1}(1-|y|)^{\frac{n+1}{2}(\sigma+\gamma)+1}=1, \nonumber  
		\end{eqnarray}
		where we used that $\frac{n+1}{2}(\sigma+\gamma)>-1$. 
		
		Finally, to treat $I_3$, let $M(r)=\max(1-r,1-|y|)$. Note that for $x$ and $y$ bounded away from zero, we have that $\theta(x,y) \approx \theta(x/|x|, y/|y|)$. Thus, introducing polar coordinates with the vector $y/|y|$ as the polar axis, we have, for $0 < r < 1$,
		\begin{multline*}
			\int_{\{|x| = r \, : \, \theta^2(x,y)>M\}}\theta(x,y)^{-(n+1)(\sigma+\tau+1)}dm_{n-1}(x) \\ \approx \int_{\sqrt{M}}^\pi \theta^{-(n+1)(\sigma+\tau+1)}\sin^{n-2}\theta d\theta 
			\approx M^{\frac{-(n+1)(\sigma+\tau+1)+n-1}{2}},
		\end{multline*}
		where we used that $ -(n+1)(\sigma+\tau+1)+n-2<-1$. Therefore, applying Tonelli's theorem we end up with the same integrals as we did for $I_1$ and $I_2$,
		\begin{eqnarray}
			I_3 &\lesssim&(1-|y|)^{\frac{n+1}{2}(\tau-\gamma)}\int_0^1 (1-r)^{\frac{n+1}{2}(\sigma+\gamma)} M(r)^{\frac{-(n+1)(\sigma+\tau+1)+n-1}{2}} dr \nonumber \\
			&\approx&(1-|y|)^{\frac{n+1}{2}(\tau-\gamma)}\int_0^{|y|} (1-r)^{\frac{n+1}{2}(\gamma-\tau-1)+\frac{n-1}{2}} dr \nonumber \\
			&  \qquad+&(1-|y|)^{-\frac{n+1}{2}(\gamma+\sigma+1)+\frac{n-1}{2}}\int_{|y|}^1 (1-r)^{\frac{n+1}{2}(\sigma+\gamma)} dr\approx 1. \nonumber
		\end{eqnarray}
		By symmetry in $x, y$ and $\sigma, \tau$, the proof is complete.
	\end{proof}

	By appealing to Lemma~\ref{Sinfty} and interpolation we obtain the following theorem, which contains Theorem~\ref{thm:mainsmooth}.
	We interpolate between $S^1$ (Theorem~\ref{S1}) and $S^2$ (Theorem~\ref{S2}), and between $S^2$ and $S^\infty$ (Lemma~\ref{Kbound}).
	\begin{theo}
		Let $\Omega$ be a bounded convex subset of $\mathbb{R}^n$ having $C^2$-boundary with non-vanishing curvatures. Then $\Omega$ is admissible, with decomposition $\{A_j^i\}$ described in Lemmas~\ref{ex}-\ref{admis}. Equip $2\Omega$ with decomposition $\{2A_j^i\}$. Then the following hold.
		\begin{enumerate}
			\item  For $\sigma,\tau>\frac{n-1}{2(n+1)}$, the extended Hankel operator $\Ha_\varphi^{\sigma,\tau}$ is bounded on $\PW(\Omega)$ if and only if $\varphi\in B_{\infty,\infty}^{\sigma+\tau}(2\Omega)$.
			\item For every $1\leq p<\infty$ and $\sigma,\tau>-\frac{1}{n+1}+\max(0,\frac{1}{2}-\frac{1}{p})$ with $\sigma+\tau>-\frac{2}{p(n+1)}+\frac{2(n-1)}{n+1}\max(0,\frac{1}{2}-\frac{1}{p})$, $\Ha_\varphi^{\sigma,\tau}$ is in the Schatten class $S^{p}(\PW(\Omega))$ if and only if $\varphi\in B_{p,p}^{\sigma+\tau+\frac{1}{p}}(2\Omega)$.
			\item For every $1\leq p<2\frac{n}{n-1}$, $\Ha_\varphi$ is in the Schatten class $S^{p}(\PW(\Omega))$ if and only if $\varphi\in B_{p,p}^{\frac{1}{p}}(2\Omega)$.
		\end{enumerate}
	\end{theo}
	
	\section{Hankel operators on the product Hardy space} \label{sec:poly}
	As mentioned in the introduction, when we choose $\Omega = \mathbb{R}^n_+$ as the first orthant, we obtain the classical product Hardy space
	$$ 
	\PW(\mathbb{R}^n_+) = H^2(\mathbb{R} \times \cdots \times \mathbb{R}) \simeq H^2(\mathbb{T}^n).
	$$
	The associated weight in this setting has a product structure,
	$$\omega_{\mathbb{R}^n_+}(x_1,...,x_n)=2^n x_1...x_n.$$ For $i=(i_1,...,i_n)\in\mathbb{Z}^n$, let us define the parallelepipeds of the decomposition as $$A_{i}=\prod_{l=1}^n(a^{i_l-1},a^{i_l+1}).$$
	It is evident that these boxes satisfy properties \ref{1}-\ref{5} (the role of $j$ is played by $i_1 + \cdots + i_n$). 
	
	In this setting, it is easy to extend Theorem \ref{S2}.
	\begin{lm}\label{S2hardy}
		When $\Omega=\mathbb{R}^n_+$, Theorem \ref{S2} holds for all $\sigma,\tau>-\frac{1}{2}$.
	\end{lm}
	\begin{proof}
		It is equivalent to prove that $$\omega_{\mathbb{R}^n_+}^{2\sigma} \ast \omega_{\mathbb{R}^n_+}^{2\tau}(x)\lesssim \omega_{\mathbb{R}^n_+}^{2\sigma + 2\tau + 1}(x/2), \qquad x\in \mathbb{R}^n_+.$$ 
		But in this case we can simply compute that
		\begin{eqnarray}
			\omega_{\mathbb{R}^n_+}^{2\sigma} \ast \omega_{\mathbb{R}^n_+}^{2\tau}(x)&\approx&\prod_{j=1}^n\int_{0}^{x_j}y_j^{2\sigma} (x_j - y_j)^{2\tau} dy_j= \prod_{j=1}^n x_j^{2\sigma+2\tau+1}\int_{0}^{1}y_j^{2\sigma} (1-y_j)^{2\tau} dy_j \nonumber  \\
			&=&B(2\sigma+1,2\tau+1)^n \prod_{j=1}^n x_j^{2\sigma+2\tau+1}, \nonumber 
		\end{eqnarray} 
		where $B$ is the Beta function.
	\end{proof}
	
	Combining Theorem \ref{S1} and Lemma \ref{S2hardy} with interpolation, and separately Lemma \ref{S2hardy} with Lemma~\ref{Sinfty}, therefore yields the sufficiency part of following theorem. Necessity is given by Theorem~\ref{converse}.
	\begin{theo}\label{Hardy}
		The following hold.
		\begin{enumerate}
			\item  For $\sigma,\tau>0$, the extended Hankel operator $\Ha_\varphi^{\sigma,\tau}$ is bounded on the product Hardy space $H^2(\mathbb{R} \times \cdots \times \mathbb{R})$ if and only if $\varphi\in B_{\infty,\infty}^{\sigma+\tau}(\mathbb{R}^n_+)$.
			\item For every $1\leq p<\infty$ and all $\sigma, \tau >\max(-\frac{1}{2},-\frac{1}{p})$, the extended Hankel operator $\Ha_\varphi^{\sigma,\tau}$ is in the Schatten class $S^{p}(H^2(\mathbb{R} \times \cdots \times \mathbb{R}))$ if and only if $\varphi\in B_{p,p}^{\sigma+\tau+\frac{1}{p}}(\mathbb{R}^n_+)$.
		\end{enumerate}
		In both cases, the Besov spaces are defined by the decomposition $\{A_i\}$, and the corresponding norms are equivalent.
	\end{theo}
	\begin{proof}
		It remains to show that the operator with kernel $$K(x,y)=\frac{\omega_{\mathbb{R}^n_+}^{\sigma}(x)\omega_{\mathbb{R}^n_+}^{\tau}(y)}{\omega_{\mathbb{R}^n_+}(\frac{x+y}{2})^{\sigma+\tau+1}},$$ is bounded whenever $\sigma,\tau>-\frac{1}{2}$. We apply Schur's test with weights $\omega_{\mathbb{R}^n_+}^{-\frac{1}{2}}(x)$ and $\omega_{\mathbb{R}^n_+}^{-\frac{1}{2}}(y)$. By symmetry, we must then verify that
		$$\sup_{y\in\mathbb{R}^n_+}\int_{\mathbb{R}^n_+}K(x,y)\left(\dfrac{\omega_{\mathbb{R}^n_+}^{-\frac{1}{2}}(x)}{\omega_{\mathbb{R}^n_+}^{-\frac{1}{2}}(y)}\right)dx<\infty.$$ 
		Due to the product structure, it suffices to show that
		$$\sup_{y\in\mathbb{R}_+}\int_{\mathbb{R}_+}\dfrac{x^\sigma y^\tau}{(x+y)^{\sigma+\tau+1}}\left(\dfrac{y}{x}\right)^{\frac{1}{2}}dx<\infty.$$
		But by the change of variable $x=yt$, we have that
		$$\int_{\mathbb{R}_+}\dfrac{x^\sigma y^\tau}{(x+y)^{\sigma+\tau+1}}\left(\dfrac{y}{x}\right)^{\frac{1}{2}}dx=\int_{\mathbb{R}_+}\dfrac{t^{\sigma-\frac{1}{2}} }{(t+1)^{\sigma+\tau+1}}dt,$$
		which converges independently of $y$, whenever $\sigma,\tau>-\frac{1}{2}$. 
	\end{proof}
    By the same proof, we also recover and extend the results of Peng \cite{MR1001657} for the cube $C_{n}=(-1,1)^{n}$.
	To reflect the tensor structure, we will index the decomposition of $C_n$ with multi-indices $i=(i_{1},...,i_{n})\in\mathbb{Z}^{n}$. When $n = 1$ and $i\in\mathbb{Z}$, we define 
	$$A_{i}=\{x\in (-1,1): x\in \sgn i(1-a^{-|i|+1},1-a^{-|i|-1})\}, \qquad i \neq 0,$$ 
	and $A_{0}=(-\frac{2}{3},\frac{2}{3})$. When $n > 1$ we define
	$$A_{(i_{1},i_{2},...,i_{n})}=A_{i_{1}}\times A_{i_{2}}\times ... \times A_{i_{n}}.$$
	It is evident that the decomposition $\{A_i\}$ satisfies all properties \ref{1}-\ref{5} (the role of $j$ is played by $-j=|i| = |i_1| + \cdots + |i_n|$). 
	\begin{theo}\label{Peng}
		For $1\leq p < \infty$ and $\sigma,\tau>\max(-\frac{1}{2},-\frac{1}{p})$ with $\sigma+\tau>-\frac1p$, the extended Hankel operator $\Ha_{\varphi}^{\sigma,\tau}$ is in the Schatten class $S^{p}(\PW(C_{n}))$ if and only if $\varphi\in B_{p,p}^{\sigma+\tau+\frac{1}{p}}(2C_{n})$, and the corresponding norms are equivalent, where the Besov space is defined with respect to the decomposition $\{2A_{i}\}$.
	\end{theo}
	Essentially the same proof furthermore applies to all simple polytopes $\Omega \subset \mathbb{R}^n$, which are, by definition, locally affinely isomorphic to the cube $C_n$ at every point of $\partial \Omega$. An appropriate decomposition $\{A_{j}^{i}\}$ for $\Omega$ is obtained from the decomposition of the cube $C_n$, by locally applying the corresponding affine transformations. We thus obtain the following generalization of Theorem \ref{Peng}.
	\begin{theo} \label{thm:mainpoly}
		Let $P$ be a simple polytope in $\mathbb{R}^{n}$. For every $1\leq p < \infty$ and $\sigma,\tau>\max(-\frac{1}{2},-\frac{1}{p})$ with $\sigma+\tau>-\frac1p$, the extended Hankel operator $\Ha_{\varphi}^{\sigma,\tau}$ is in $S^{p}(\PW(P))$ if and only if $\varphi\in B_{p,p}^{\sigma+\tau+\frac{1}{p}}(2P)$, with equivalence of norms.
	\end{theo}

	\bibliographystyle{plain}
	\bibliography{BibSchatten1}

\begin{thebibliography}{10}

\bibitem{MR2327289}
A.~Baddeley, I.~B\'{a}r\'{a}ny, R.~Schneider, and W.~Weil.
\newblock {\em Stochastic geometry}, volume 1892 of {\em Lecture Notes in
  Mathematics}.
\newblock Springer-Verlag, Berlin, 2007.
\newblock Lectures given at the C.I.M.E. Summer School held in Martina Franca,
  September 13--18, 2004, With additional contributions by D. Hug, V. Capasso
  and E. Villa, Edited by W. Weil.

\bibitem{bampouras2023failure}
Konstantinos Bampouras.
\newblock On the failure of the {N}ehari theorem for {P}aley-{W}iener spaces.
\newblock {\em J. Math. Anal. Appl.}, 536(1):Paper No. 128165, 9, 2024.

\bibitem{BampourasPerfekt2026}
Konstantinos Bampouras and Karl-Mikael Perfekt.
\newblock {B}oundary-{W}eighted {F}ourier {I}nequalities for {C}onvex
  {D}omains.
\newblock {\em arXiv:2608.19806}, 2026.

\bibitem{MR986636}
I.~B\'{a}r\'{a}ny and D.~G. Larman.
\newblock Convex bodies, economic cap coverings, random polytopes.
\newblock {\em Mathematika}, 35(2):274--291, 1988.

\bibitem{MR2327291}
Imre B\'ar\'any.
\newblock Random polytopes, convex bodies, and approximation.
\newblock In {\em Stochastic geometry}, volume 1892 of {\em Lecture Notes in
  Math.}, pages 77--118. Springer, Berlin, 2007.

\bibitem{MR482275}
J\"{o}ran Bergh and J\"{o}rgen L\"{o}fstr\"{o}m.
\newblock {\em Interpolation spaces. {A}n introduction}.
\newblock Grundlehren der Mathematischen Wissenschaften, No. 223.
  Springer-Verlag, Berlin-New York, 1976.

\bibitem{BP23}
Ole~Fredrik Brevig and Karl-Mikael Perfekt.
\newblock The {N}ehari problem for the {P}aley-{W}iener space of a disc.
\newblock {\em J. Geom. Anal.}, 33(1):Paper No. 16, 7, 2023.

\bibitem{ChaFef85}
Sun-Yung~A. Chang and Robert Fefferman.
\newblock Some recent developments in {F}ourier analysis and {$H^p$}-theory on
  product domains.
\newblock {\em Bull. Amer. Math. Soc. (N.S.)}, 12(1):1--43, 1985.

\bibitem{MR1415032}
Mischa Cotlar and Cora Sadosky.
\newblock Two distinguished subspaces of product {BMO} and {N}ehari-{AAK}
  theory for {H}ankel operators on the torus.
\newblock {\em Integral Equations Operator Theory}, 26(3):273--304, 1996.

\bibitem{FergLac02}
Sarah~H. Ferguson and Michael~T. Lacey.
\newblock A characterization of product {BMO} by commutators.
\newblock {\em Acta Math.}, 189(2):143--160, 2002.

\bibitem{FergSad00}
Sarah~H. Ferguson and Cora Sadosky.
\newblock Characterizations of bounded mean oscillation on the polydisk in
  terms of {H}ankel operators and {C}arleson measures.
\newblock {\em J. Anal. Math.}, 81:239--267, 2000.

\bibitem{HTV21}
Irina Holmes, Sergei Treil, and Alexander Volberg.
\newblock Dyadic bi-parameter repeated commutator and dyadic product {BMO}.
\newblock {\em arXiv:2101.00763}, 2021.

\bibitem{MR4180684}
Daniel Hug and Wolfgang Weil.
\newblock {\em Lectures on convex geometry}, volume 286 of {\em Graduate Texts
  in Mathematics}.
\newblock Springer, Cham, [2020] \copyright 2020.

\bibitem{MR924766}
Svante Janson and Jaak Peetre.
\newblock Paracommutators---boundedness and {S}chatten-von {N}eumann
  properties.
\newblock {\em Trans. Amer. Math. Soc.}, 305(2):467--504, 1988.

\bibitem{John1948}
Fritz John.
\newblock Extremum problems with inequalities as subsidiary conditions.
\newblock In {\em Studies and Essays Presented to R. Courant on his 60th
  Birthday}, pages 187--204. Interscience, New York, 1948.

\bibitem{LLW23}
Michael~T. Lacey, Ji~Li, and Brett~D. Wick.
\newblock Schatten class estimates for paraproducts in multi-parameter setting
  and applications.
\newblock {\em 2303.15657}, 2023.

\bibitem{Nik02}
Nikolai~K. Nikolski.
\newblock {\em Operators, functions, and systems: an easy reading. {V}ol. 1},
  volume~92 of {\em Mathematical Surveys and Monographs}.
\newblock American Mathematical Society, Providence, RI, 2002.
\newblock Hardy, Hankel, and Toeplitz, Translated from the French by Andreas
  Hartmann.

\bibitem{Peller88}
Vladimir~V. Peller.
\newblock Wiener-{H}opf operators on a finite interval and {S}chatten-von
  {N}eumann classes.
\newblock {\em Proc. Amer. Math. Soc.}, 104(2):479--486, 1988.

\bibitem{MR1949210}
Vladimir~V. Peller.
\newblock {\em Hankel operators and their applications}.
\newblock Springer Monographs in Mathematics. Springer-Verlag, New York, 2003.

\bibitem{MR953994}
Li~Zhong Peng.
\newblock Hankel operators on the {P}aley-{W}iener space in disk.
\newblock In {\em Miniconferences on harmonic analysis and operator algebras
  ({C}anberra, 1987)}, volume~16 of {\em Proc. Centre Math. Anal. Austral. Nat.
  Univ.}, pages 173--180.

\bibitem{MR1001657}
Li~Zhong Peng.
\newblock Hankel operators on the {P}aley-{W}iener space in {$\bold R^d$}.
\newblock {\em Integral Equations Operator Theory}, 12(4):567--591, 1989.

\bibitem{MR2097606}
Sandra Pott and Martin~P. Smith.
\newblock Paraproducts and {H}ankel operators of {S}chatten class via
  {$p$}-{J}ohn-{N}irenberg theorem.
\newblock {\em J. Funct. Anal.}, 217(1):38--78, 2004.

\bibitem{MR674875}
Richard Rochberg.
\newblock Trace ideal criteria for {H}ankel operators and commutators.
\newblock {\em Indiana Univ. Math. J.}, 31(6):913--925, 1982.

\bibitem{MR878246}
Richard Rochberg.
\newblock Toeplitz and {H}ankel operators on the {P}aley-{W}iener space.
\newblock {\em Integral Equations Operator Theory}, 10(2):187--235, 1987.

\bibitem{MR1194970}
M.~Schmuckenschl\"{a}ger.
\newblock The distribution function of the convolution square of a convex
  symmetric body in {${\bf R}^n$}.
\newblock {\em Israel J. Math.}, 78(2-3):309--334, 1992.

\bibitem{MR1119928}
Carsten Sch\"{u}tt.
\newblock The convex floating body and polyhedral approximation.
\newblock {\em Israel J. Math.}, 73(1):65--77, 1991.

\bibitem{SW90}
Carsten Sch\"utt and Elisabeth Werner.
\newblock The convex floating body.
\newblock {\em Math. Scand.}, 66(2):275--290, 1990.

\bibitem{MR837521}
Dan Timotin.
\newblock A note on {$C_p$} estimates for certain kernels.
\newblock {\em Integral Equations Operator Theory}, 9(2):295--304, 1986.

\end{thebibliography}
	
\end{document}